\documentclass[10pt,oneside,reqno]{amsart}

\usepackage[utf8]{inputenc}
\usepackage[T1]{fontenc}
\usepackage[russian,ngerman,english]{babel}
\usepackage{lmodern}

\usepackage{hyperref}

\usepackage{amsmath}
\usepackage{amsthm}
\usepackage{amsfonts}
\usepackage{amssymb}
\usepackage{amscd}
\usepackage{amsbsy}

\usepackage{pdfpages}
\usepackage{fancyhdr}
\usepackage{geometry}
\usepackage{setspace}
\usepackage{xcolor}
\usepackage{pifont}

\usepackage{graphicx}
\usepackage{tabularx}
\usepackage{epsfig}
\usepackage[all]{xy}
\usepackage{tikz}
\usetikzlibrary{matrix}

\usepackage{fixmath}

\usepackage{appendix}

\usepackage{enumerate}

\usepackage[sort, numbers]{natbib}
\usepackage{bibentry}

\newtheorem{theorem}{Theorem}[section]
\newtheorem{lemma}[theorem]{Lemma}

\newtheorem{corollary}[theorem]{Corollary}

\theoremstyle{definition}
\newtheorem{definition}[theorem]{Definition}

\newtheorem*{remark}{Remark}

\newtheorem{example}[theorem]{Example}

\numberwithin{equation}{section}

\newcommand{\R}{{\mathbb R}}
\renewcommand{\C}{{\mathbb C}}
\newcommand{\D}{{\mathbb D}}

\newcommand{\E}{{\mathsf E}}

\newcommand{\cD}{{\mathcal{D}}}

\newcommand{\cM}{{\mathcal{M}}}

\newcommand{\cP}{{\mathcal{P}}}

\newcommand{\cS}{{\mathcal{S}}}

\newcommand{\cU}{{\mathcal{U}}}

\renewcommand{\l}{\lambda}
\renewcommand{\a}{\alpha}

\renewcommand{\t}{{\theta}}

\renewcommand{\Re}{\operatorname{Re}}
\renewcommand{\Im}{\operatorname{Im}}

\newcommand{\ess}{\text{\rm{ess}}}

\newcommand{\bbD}{\mathbb{D}}
\newcommand{\bbN}{\mathbb{N}}
\newcommand{\bbZ}{\mathbb{Z}}

\newcommand{\bbR}{\mathbb{R}}
\newcommand{\bbC}{\mathbb{C}}

\newcommand{\AC}{\mathrm{AC}}
\newcommand{\loc}{\mathrm{loc}}
\newcommand{\unif}{\mathrm{unif}}

\DeclareMathOperator{\intt}{int}

\DeclareMathOperator{\wto}{\stackrel{w}{\to}}

\newcommand{\inv}{^{-1}}

\newcommand{\smat}{\begin{pmatrix}}
\newcommand{\fmat}{\end{pmatrix}}

\newcommand{\vertiii}[1]{{\left\vert\kern-0.25ex\left\vert\kern-0.25ex\left\vert #1 
    \right\vert\kern-0.25ex\right\vert\kern-0.25ex\right\vert}}

\allowdisplaybreaks

\title{Stahl--Totik regularity for Dirac operators}

\thanks{B.E.\ was supported by Austrian Science Fund FWF, project no: J 4138-N32.}
\thanks{E.G.\ was supported in part by NSF grant DMS--1745670.}
\thanks{M.L.\ was supported in part by NSF grant DMS--1700179.}

\author{Benjamin~Eichinger, Ethan~Gwaltney, Milivoje~Luki\'c}
\begin{document}

\address{Institute of Analysis, Johannes Kepler University of Linz, 4040 Linz, Austria.}
\email{benjamin.eichinger@rice.edu}
\address{Department of Mathematics, Rice University MS-136, Box 1892,
Houston, TX 77251-1892, USA.}
\email{ethan.gwaltney@rice.edu}
\address{Department of Mathematics, Rice University MS-136, Box 1892,
Houston, TX 77251-1892, USA.}
\email{milivoje.lukic@rice.edu}

\begin{abstract}
We develop a theory of regularity for Dirac operators with uniformly locally square-integrable operator data. This is motivated by Stahl--Totik regularity for orthogonal polynomials and by recent developments for continuum Schr\"odinger operators, but contains significant new phenomena. We prove that the symmetric Martin function at $\infty$ for the complement of the essential spectrum has the two-term asymptotic expansion $\Im \left( z - \frac{b}{2 z}\right) + o(\frac 1z)$ as $z \to i \infty$, which is seen as a thickness statement for the essential spectrum.   The constant $b$ plays the role of a renormalized Robin constant and enters a universal inequality involving the lower average $L^2$-norm of the operator data. However, we show that regularity of Dirac operators is not precisely characterized by a single scalar equality involving $b$ and is instead characterized by a family of equalities. This work also contains a sharp Combes--Thomas estimate (root asymptotics of eigensolutions), a study of zero counting measures, and applications to ergodic and decaying operator data.
\end{abstract}

\maketitle

\section{Introduction}

We study half-line Dirac operators, which commonly appear in one of the two forms
\begin{align}
\Lambda_\varphi & = 
			\begin{pmatrix} 
				i & 0 \\ 0 & - i 
			\end{pmatrix} 
			\frac{d}{dx} + \begin{pmatrix} 0 & \varphi \\ \bar{\varphi} & 0 \end{pmatrix}  
\label{ZSO} \\
L_\varphi & =   	\begin{pmatrix}
		0 & -1 \\
		1 & 0
	\end{pmatrix} \frac{d}{dx} -
			\begin{pmatrix}
				\Re \varphi & \Im \varphi \\
				\Im \varphi & -\Re \varphi
			\end{pmatrix}  
			\label{DO}
\end{align}
While the form \eqref{DO} is more classical \cite{LevSar}, both forms are common in the literature; in the form \eqref{ZSO}, they are often called Zakharov--Shabat operators, due to their appearance in the Lax pair representation of the nonlinear Schr\"odinger equation \cite{GreKap}. The two forms are related by a simple matrix conjugation, so both their spectral properties and the behavior of their eigensolutions are related trivially; it will be very practical to work with the form \eqref{ZSO} and note that all conclusions below apply for both. We give more information in Section~\ref{sectionDiracBackground}.

The goal of this paper is to develop a theory of regularity for Dirac operators. This is inspired by the theory of Stahl--Totik regularity \cite{Ullman72,StahlTotik92}, which provides deep connections between orthogonal polynomials and potential theory by comparing the root asymptotics and asymptotic behavior of the zero distribution of orthogonal polynomials to the  Green function and equilibrium measure for the complement of the support of a compactly supported orthogonality measure. A counterpart of regularity was recently developed for continuum Schr\"odinger operators $-\frac{d^2}{dx^2} +V$ \cite{EL}; their spectra are semibounded, so infinity is a boundary point of the complement of the spectrum and it is shown that the theory relies on the Martin function at infinity for the complement of the spectrum. This suggests that a similar theory may exist for other systems with unbounded spectra and motivates the current paper. There is a general expectation in spectral theory that various classes of difference and differential operators exhibit the same phenomena, and an opposite expectation in integrable systems that every infinite dimensional integrable system is different. For the theory described here, the second expectation is more accurate. There will be many differences from the Schr\"odinger case and we will comment on them along the way; some will be technical obstacles and others will be new phenomena. We will see that the exact form of the operator and the local regularity of $\varphi$ play an important role.

Recall that Schr\"odinger and Dirac operators are central objects for the Lax pair representations of KdV and defocusing NLS hierarchies. For periodic operators, these hierarchies are represented as Poisson structures  \cite{KapPos,GreKap}. The KdV hierarchy in the periodic case starts with the Casimir element $\int_0^1 V(t) dt$; in the half-line setting, the Ces\`aro average of $V$ is central to regularity of Schr\"odinger operators \cite{EL}. 
The NLS hierarchy in the periodic case starts with the mass $\int_0^1 \lvert \varphi(t)\rvert^2 dt$; accordingly, it will be natural to work under a uniform local $L^2$ condition on $\varphi$, but we will see that $L^2$ Ces\`aro averages of $\varphi$ do not capture regularity in general, and we will describe the necessary modifications.

The operator data $\varphi: [0,\infty) \to \bbC$ is a complex valued function, and we always assume that it obeys the uniform local $L^2$ boundedness condition
\begin{equation}\label{varphiL2locunif}
\sup_{x\ge 0} \int_x^{x+1} \lvert \varphi(t) \rvert^2 \,dt < \infty.
\end{equation}
This is sufficient to guarantee that $\Lambda_\varphi$ has a regular endpoint at $0$ and is limit point at $+\infty$, so $\Lambda_\varphi$ is made into an unbounded self-adjoint operator on the Hilbert space $L^2([0,\infty),\bbC^2) \cong L^2([0,\infty)) \oplus L^2([0,\infty))$ by setting the domain
\begin{equation}\label{15jun1}
D(\Lambda_\varphi) = \{ f \in H^1([0,\infty),\bbC^2) \mid f_1(0) = f_2(0) \}.
\end{equation}
The essential spectrum $\E = \sigma_\ess(\Lambda_\varphi)$ is unbounded above and below, since $\Lambda_\varphi - \Lambda_0$ is an operator-bounded perturbation of $\Lambda_0$ with relative bound $0$ (in the sense of \cite[Section X.2]{RS2}) when \eqref{varphiL2locunif} holds. 

For any unbounded closed set $\E \subset \bbR$, the domain $\Omega = \bbC \setminus \E$ is called a Denjoy domain. Note that $\Omega$ has $\infty$ as a boundary point. When the general notion of Martin boundary \cite{Martin41,ArmGar01} is applied to Denjoy domains \cite{Be80}, it associates to $\infty$ a cone of dimension $1$ or $2$ of positive harmonic functions in $\Omega$ which are bounded on bounded sets and vanish at every Dirichlet-regular point of $\E$. In both cases, if we impose an additional symmetry condition $M(\bar z) = M(z)$, this symmetric Martin function $M$ is determined uniquely up to normalization.

The dimension of the cone at $\infty$ is determined by the Akhiezer--Levin condition
\begin{equation}\label{15jun3}
\lim_{y\to  \infty} \frac{M(iy)}{y} > 0
\end{equation}
(by general principles, the limit exists with a value in $[0,\infty)$): the cone is $2$-dimensional if \eqref{15jun3} holds and $1$-dimensional  if \eqref{15jun3} fails. For Akhiezer--Levin sets, we will normalize the symmetric Martin function so that the limit in \eqref{15jun3} is equal to $1$ and denote by $M_\E$ the symmetric Martin function with this normalization,
\begin{equation}\label{MartinFuncNormalization}
\lim\limits_{y\to\infty}\frac{M_\E(iy)}{y}=1.
\end{equation}
The case $\E = \bbR$ does not fit neatly in the above discussion since $\bbC \setminus \bbR$ is not connected, but it is easily seen for each claim that $M_{\bbR}(z) = \lvert \Im z \rvert$ is the correct analog.

Schr\"odinger operators with bounded potentials can have very thin spectra, of zero Lebesgue measure and zero Hausdorff dimension \cite{DFL17,DFG19}; we expect that the same behaviors are possible for Dirac operators. However, in the language of potential theory and Martin functions, we obtain the following universal thickness conditions on the essential spectrum:

\begin{theorem}\label{thm11}
For any $\varphi$ obeying \eqref{varphiL2locunif} and $\E = \sigma_\ess(\Lambda_\varphi)$, the domain $\Omega= \bbC \setminus \E$ is not a polar set, $\infty$ is a Dirichlet-regular boundary point for $\Omega$, $\Omega$ obeys the Akhiezer--Levin condition, and there exists $b_\E \ge 0$ such that the Martin function has the asymptotic behavior
\begin{equation}\label{MartinFunc2termexp}
M_\E(z) = \Im \left( z - \frac{b_\E}{2z} \right) + o\left( \frac 1{\lvert z \rvert}\right),
\end{equation}
as $z \to \infty$, $\arg z \in [\delta,\pi -\delta]$, for any $\delta > 0$.
\end{theorem}

Note that Theorem~\ref{thm11} provides an alternative proof that $\E= \sigma_\ess(\Lambda_\varphi)$ is unbounded above and below when \eqref{varphiL2locunif} holds: a set $\E$ which obeys \eqref{15jun3} must be unbounded above and below.

Each conclusion of Theorem~\ref{thm11} is strictly stronger than the previous. In particular, existence of the two-term asymptotics is not a general consequence of the Akhiezer--Levin condition: for instance, \cite[Example 2.6]{EL} gives sets $\E$ for which the Martin function obeys $M_\E(iy) = y + c + o(1)$ as $y\to\infty$, with $c$ an arbitrary real constant. Although the leading behavior of Martin functions is well studied, the two-term asymptotics \eqref{MartinFunc2termexp} were not previously considered. The conclusions of Theorem~\ref{thm11} are expressions about the thickness of the set $\E$, and the constant $b_\E$ serves as a measure of size of $\E$, as illustrated by the following lemma:
\begin{lemma}\label{lemmamonotonicity}
Let $\E_1 \subset \E_2 \subset \bbR$ be closed unbounded sets.  If the Martin function for $\E_1$ obeys asymptotics of the form \eqref{MartinFunc2termexp}, then so does the Martin function for $\E_2$. Moreover, $b_{\E_2} \le b_{\E_1}$, and equality holds if and only if $\E_2 \setminus \E_1$ is a polar set.
\end{lemma}
Note the direction in this monotonicity statement: $b_\E$ decreases as $\E$ increases. Thus, $b_\E$ is an inverse measure of the size of $\E$ (or, a measure of the size of the gaps), so the following theorem should be understood as a lower bound on the size of the spectrum in terms of the operator data $\varphi$:

\begin{theorem} \label{thm12}
If $\varphi$ obeys \eqref{varphiL2locunif} and $\E = \sigma_\ess(\Lambda_\varphi)$, then
\begin{equation}\label{29aug3}
b_\E \le \liminf_{x\to\infty} \frac 1x \int_0^x \lvert \varphi(t) \rvert^2 \,dt.
\end{equation}
\end{theorem}

For any $z \in \bbC$, consider the Dirichlet solution $U = \binom{u_1}{u_2} \in \AC_\loc([0,\infty),\bbC^2)$ which solves the initial value problem
\[
\Lambda_\varphi U = z U, \qquad U(0,z) = \begin{pmatrix} 1 \\ 1 \end{pmatrix}.
\]
Statements about the exponential decay of resolvent kernels, or exponential growth of eigensolutions linearly independent from the Weyl solutions, are often called Combes--Thomas estimates. The following is a sharp Combes--Thomas estimate for Dirac operators:

\begin{theorem} \label{thm13}
If $\varphi$ obeys \eqref{varphiL2locunif} and $\E = \sigma_\ess(\Lambda_\varphi)$, then
\begin{equation}\label{29aug4}
M_\E(z) \le \liminf_{x\to\infty} \frac 1x \log \lVert U(x,z) \rVert, \qquad \forall z \in \bbC \setminus\bbR.
\end{equation}
\end{theorem}

So far, the results closely resemble those for Schr\"odinger operators, but there are already differences in the proofs. Unboundedness of $\E$ both above and below requires some changes to the Martin function arguments, since there is no longer a consistent choice of normalization point below the bottom of the essential spectrum; some of these arguments are given in Section~\ref{sectionMartin}. However, the main additional obstacles in this paper come from the difficulty in obtaining precise asymptotics for Dirichlet solutions with the necessary uniformity. For Schr\"odinger operators, asymptotics of fundamental solutions with polynomial expansions in $1/\sqrt{-z}$ follow easily from the Volterra series \cite{PosTru}, because one of the fundamental solutions for the free Schr\"odinger operator is $\sinh( \sqrt{-z} x) / \sqrt{-z}$ with the explicit $\sqrt{-z}$ in the denominator. For Dirac operators, it is more difficult to obtain precise control of the eigensolutions. The added difficulties are explicitly noted in \cite{ClaGes02} where detailed asymptotics of the Weyl $m$-functions are derived; we should note that the results of  \cite{ClaGes02} would be insufficient for our purposes, since their more detailed estimates are provided at Lebesgue points of $\varphi$ and do not offer the uniformity in $x$ needed for our method. In other references such as \cite{GreKap,KapSchTop}, multi-term expansions of fundamental solutions are derived under the stronger assumption that $\varphi$ is locally in the Sobolev space $H^1$, and the additional smoothness is used to extract factors of $1/z$ by integration by parts. We will avoid both of these obstacles by working directly with Weyl solutions and extracting an explicit factor of $1/z$ through a convolution representation that seems to be new.

Based on Theorems~\ref{thm12}, \ref{thm13} and regularity in other settings, it would be natural to conjecture that equality (with $\lim$ instead of $\liminf$)   holds in \eqref{29aug4} if and only if
\begin{equation}\label{1dec1}
\lim_{x\to\infty} \frac 1x \int_0^x \lvert \varphi(t) \rvert^2 \,dt = b_\E.
\end{equation}
However, this is false:

\begin{example} \label{xmpl11}
Let $\varphi(x) = \sin(x^\alpha)$ for some $\alpha > 1$. Then $\E = \sigma_\ess(\Lambda_\varphi) = \bbR$, $b_\E = 0$,
\[
\lim_{x\to\infty} \frac 1x \log \lVert U(x,z) \rVert = M_\E(z) \qquad \forall z \in \bbC \setminus \bbR,
\]
but the limit in \eqref{1dec1} is not equal to $b_\E$.
\end{example}

Not only are the expansions harder to derive, but the seemingly natural expansions have worse uniformity properties compared to Schr\"odinger operators. A key role in the proofs is played by the asymptotic behavior as $z \to \infty$ of the subsequential limits as $x \to \infty$ of a function closely related to $\frac 1x \log \lVert U(x,z) \rVert$. For Schr\"odinger operators, the analogous objects have two-term nontangential asymptotics $\Re \left( \sqrt{-z} + \frac 1x \int_0^x V(t) dt / (2 \sqrt{-z}) \right)$, and by estimates uniform in $x$, it is proved that the subsequential limit along $x_j \to \infty$ has two-term asymptotics $\Re \left( \sqrt{-z} + a / (2 \sqrt{-z}) \right)$ where $\frac 1{x_j} \int_0^{x_j} V(t) dt \to a$. For Dirac operators, the function related to $\frac 1x \log \lVert U(x,z) \rVert$ will have the two-term asymptotics $\Im \left( z - \frac 1x \int_0^x \lvert \varphi(t) \rvert^2 \,dt / z \right)$ as the leading coefficient, but those do not determine the two-term asymptotics of the subsequential limit. In other words, unlike for Schr\"odinger operators, for Dirac operators the limits implicit in these two-term asymptotics do not in general commute, as demonstrated by Example~\ref{xmpl11}.

To describe the new phenomena we use the Fourier transform on $\bbR$, with the conventions $\hat f(k) = \frac 1{\sqrt{2\pi}} \int f(x) e^{-ikx} dx$, and we define for $x > 0$ the measure on $\bbR$
\[
d\sigma_x(k) = \frac{1}{x} \lvert \widehat{ ( \varphi \chi_{[0,x]} ) } (k) \rvert^2 \,dk.
\]
We emphasize that the measures $\sigma_x$ are \emph{not} the familiar zero counting measures for spectral problems on intervals, they are not Riesz measures associated with subharmonic functions on $\bbC$, and they are not spectral measures for the operator. The measures $\sigma_x$ directly correspond to the operator data $\varphi$. They don't have an analog for Schr\"odinger operators; alternatively, it could be argued that their Schr\"odinger analogs are merely the numbers $\frac 1x \int_0^x V(t) dt$.  Just as those numbers in the Schr\"odinger case are confined to a compact interval by uniform local integrability of $V$, we will need to consider precompactness of the measures $\sigma_x$ in an appropriate topology.

By unitarity of the Fourier transform,
\[
\sigma_x(\bbR) = \frac 1x \int_0^x \lvert \varphi(t) \rvert^2 \,dt.
\]
For $x \ge 1$, the measures are uniformly bounded: \eqref{varphiL2locunif} implies that $\sup_{x \ge 1} \sigma_x(\bbR) < \infty$.  On these measures, we will use the notion of vague convergence, dual to the set $C_0(\bbR)$ of continuous functions $f$ on $\bbR$ with $\lim_{x\to \pm \infty} f(x) = 0$.  Denote by $\cS(\varphi)$ the set of limits of $\sigma_x$ as $x \to \infty$ in the vague topology. This set is nonempty by a consequence of the Banach--Alaoglu theorem, and the Portmanteau theorem implies
\begin{equation}\label{29aug1}
\inf_{\sigma \in \cS(\varphi)} \sigma(\bbR) \le \liminf_{x\to\infty}  \frac 1x \int_0^x \lvert \varphi(t) \rvert^2 \,dt
\end{equation}
 (see Section~\ref{sectionvague}). Strict inequality is possible here and this is at the core of why Theorem~\ref{thm12} is not optimal and Example~\ref{xmpl11} is possible. Theorem~\ref{thm12} can be improved to:

\begin{theorem}\label{thmunivineq2}
If $\varphi$ obeys \eqref{varphiL2locunif} and $\E = \sigma_\ess(\Lambda_\varphi)$, then
\begin{equation}\label{20jun2}
b_\E \le \inf_{\sigma \in \cS(\varphi)} \sigma(\bbR).
\end{equation}
\end{theorem}

We now define regularity for a Dirac operator by

\begin{definition}
The Dirac operator $\Lambda_\varphi$ is regular if
\[
b_\E = \sigma(\bbR), \qquad \forall \sigma \in \cS(\varphi).
\]
\end{definition}

We consider this to be the correct definition because it has equivalent characterizations in terms of growth rates of Dirichlet solutions, analogous to Stahl--Totik regularity of orthogonal polynomials and to \cite[Theorem 1.5]{EL}, as seen in the following theorem. Recall that harmonic measure for the domain $\bbC \setminus \E$ is, for any $z_0 \in \bbC \setminus \E$, a probability measure on $\E\cup\{\infty\}$ associated with the Dirichlet problem on $\Omega$. Since harmonic measures for different reference points $z_0$ are mutually absolutely continuous, we don't need to specify $z_0$ in the following claim.

\begin{theorem}
\label{RegularityTFAE}
If $\varphi$ obeys \eqref{varphiL2locunif} and $\E = \sigma_\ess(\Lambda_\varphi)$, the following are equivalent:
\begin{enumerate}[(i)]
\item $b_\E = \sigma(\bbR)$ for all $\sigma \in \cS(\varphi)$;
\item For every Dirichlet-regular $z\in \E$, $\limsup_{x\to\infty} \frac 1x \log \lVert U(x,z) \rVert \le 0$;
\item For a.e.\ $z \in \E$ with respect to harmonic measure for the domain $\bbC \setminus \E$, ${\limsup_{x\to\infty} \frac 1x \log \lVert U(x,z) \rVert \le 0}$;
\item For some $z \in \bbC_+$, $\limsup_{x\to \infty} \frac 1x \log \lVert U(x,z) \rVert \le M_\E(z)$;
\item For all $z \in \bbC$, $\limsup_{x\to \infty} \frac 1x \log \lVert U(x,z) \rVert \le M_\E(z)$;
\item $\lim_{x\to \infty} \frac 1x \log \lVert U(x,z) \rVert = M_\E(z)$ uniformly on compact subsets of $\bbC \setminus \bbR$.
\end{enumerate}
\end{theorem}

Intuitively, strict inequality in \eqref{29aug1} happens when part of the $L^2$-average escapes through higher Fourier modes as $x \to \infty$. It may be more transparent to formulate results in a form which doesn't use vague convergence, so we express Theorem~\ref{thmunivineq2} and the characterization of regularity in terms of local averages of $\varphi$:

\begin{theorem}\label{thmlocalaverages}
If $\varphi$ obeys \eqref{varphiL2locunif} and $\E = \sigma_\ess(\Lambda_\varphi)$, then 
\begin{equation}\label{20jun1}
b_\E \le \sup_{\epsilon > 0} \liminf_{x \to\infty}  \frac 1x \int_0^x \left\lvert \frac 1{\epsilon} \int_{t}^{t+\epsilon}  \varphi(s)\,ds \right\rvert^2 \,dt.
\end{equation}
Moreover, the Dirac operator is regular if and only if for every $\epsilon > 0$,
\begin{equation}\label{22nov3}
\limsup_{x \to \infty}  \frac 1x \int_0^x \left\lvert \frac 1{\epsilon} \int_{t}^{t+\epsilon}  \varphi(s)\,ds \right\rvert^2 \,dt \le b_\E.
\end{equation}
\end{theorem}

The case when constants go to $0$ can be particularly well understood, because $b_\E \ge 0$ with equality if and only if $\E = \bbR$. Thus, the following corollary is immediate:

\begin{corollary} \label{corollarysigmaR}
For any Dirac operator with $\varphi \in L^2_{\loc,\unif}([0,\infty))$, if
\begin{equation}\label{22nov5}
\liminf_{x \to\infty}  \frac 1x \int_0^x \left\lvert \frac 1{\epsilon} \int_{t}^{t+\epsilon}  \varphi(s)\,ds \right\rvert^2 \,dt = 0
\end{equation}
for all $\epsilon > 0$, then $\sigma_\ess(\Lambda_\varphi) = \bbR$.
\end{corollary}

Classical Blumenthal--Weyl type results state that a decaying potential does not affect the essential spectrum;  Corollary~\ref{corollarysigmaR} is a Ces\`aro type generalization of such a result. Of course, note that \eqref{22nov5} follows from $\liminf_{x\to \infty} \frac 1x \int_0^x \lvert \varphi(t) \rvert^2 \,dt = 0$.

It is now easy to formulate perturbative criteria for regularity for classes of perturbations which don't affect essential spectrum and don't affect the averages \eqref{22nov3}. The opposite direction is more subtle: the question of whether regularity implies some form of decay was first considered by Simon~\cite{Simon09}, who proved that regularity for a Jacobi matrix with spectrum $[-2,2]$ implies a certain Ces\`aro type decay. In the Schr\"odinger case the corresponding statement turns out to be false: if a regular Schr\"odinger operator has essential spectrum $[0,\infty)$, this does not imply that  $\frac 1x \int_0^x \lvert V(t) \rvert dt\to 0$ as $x \to \infty$ \cite{EL}. For Dirac operators, we not only prove the correct analog of Simon's result,  we obtain an ``if and only if'' characterization of regular operators with $\E = \bbR$.

\begin{corollary} \label{corCesaroNevai}
For any $\varphi$ which obeys \eqref{varphiL2locunif}, the following are equivalent:
\begin{enumerate}[(i)]
\item $\sigma(\Lambda_\varphi) = \bbR$ and $\Lambda_\varphi$ is regular;
\item For all $\epsilon > 0$,
\begin{equation}\label{23nov1}
\lim_{x \to \infty}  \frac 1x \int_0^x \left\lvert \frac 1{\epsilon} \int_{t}^{t+\epsilon}  \varphi(s)\,ds \right\rvert^2 \,dt = 0.
\end{equation}
\end{enumerate}
\end{corollary}

Although regularity on $\bbR$ does not imply decay of $L^2$-averages, (ii) is a weak $L^2$ Ces\`aro decay statement. Let us illustrate one more possibility: 

\begin{example} \label{xmpl12}
Let $\varphi(x) = \chi_B(x) \sin(x^\alpha)$ where $B = \cup_{k=0}^\infty [q^{2k}, q^{2k+1})$ 
for some $\alpha, q > 1$. Then  $\Lambda_\varphi$ is regular and $\sigma(\Lambda_\varphi) = \bbR$, but  $\lim_{x\to\infty} \frac 1x \int_0^x \lvert \varphi(t)\rvert^2 dt$ does not exist.
\end{example}

 If the family $\{ \sigma_x \mid x \ge 1\}$ is tight (i.e. its tails are uniformly decaying), then equality holds in \eqref{29aug1} and these distinctions vanish.  In particular, we obtain this if the $L^2$-boundedness \eqref{varphiL2locunif} of length $1$ restrictions of $\varphi$ is replaced by precompactness:

\begin{corollary}\label{corL2precompact} 
Assume that the set $\{ \varphi(\cdot - x)\rvert_{(0,1)} \mid x \ge 0 \}$ is precompact in $L^2((0,1))$. Then the operator $\Lambda_\varphi$ is regular if and only if \eqref{1dec1} holds.
\end{corollary}

For instance, if $\varphi$ is periodic and locally $L^2$, the set of restrictions $\varphi(\cdot - x)\rvert_{(0,1)}$ is continuously parametrized by $x \in \bbR / T\bbZ$ where $T$ is the period. As a continuous image of the compact circle, this family is compact. By Floquet theory and Theorem~\ref{RegularityTFAE}(iii), periodic operators  are regular, so by Corollary~\ref{corL2precompact}, they obey \eqref{1dec1}.

Precompactness can also be a consequence of additional regularity. On an open interval $I$, the Sobolev spaces $H^\gamma(I) = W^{\gamma,2}(I)$ are defined for any $\gamma > 0$  \cite{DPV}; in particular, the fractional Sobolev spaces $H^\gamma(I)$ for $\gamma \in (0,1)$ are defined by the norm
\begin{equation}\label{1dec3}
\lVert \varphi \rVert_{H^\gamma(I)}^2  = \int_I \lvert \varphi(x) \rvert^2 \,dx + \iint_{I\times I} \frac{ \lvert \varphi(x) - \varphi(y) \rvert^2}{\lvert x - y\rvert^{1+2\gamma}}\,dx \,dy.
\end{equation}
By a version of a Riesz--Frechet--Kolmogorov compact embedding theorem \cite[Lemma 10]{PSV}, \cite[Theorem 7.1]{DPV}, bounded subsets of $H^\gamma(I)$ are precompact subsets of $L^2(I)$, so Corollary~\ref{corL2precompact} implies:

\begin{corollary}\label{corHgamma}
Assume that for some $\gamma \in (0,\infty)$, $\varphi$ obeys $\sup_{x\ge 0} \lVert \varphi  \rVert_{H^\gamma((x,x+1))} < \infty$. Then the Dirac operator $\Lambda_\varphi$ is regular if and only if \eqref{1dec1} holds.
\end{corollary}

We conclude that local $L^2$ regularity is the correct setting for this theory; weaker local conditions wouldn't work in many $L^2$-based estimates, and stronger local conditions would have missed the surprisingly rich phenomena described above.

For any $x > 0$, let $\rho_x$ denote the renormalized zero counting measure 
\[
\rho_x = \frac 1x \sum_{z: u_1(x,z) = u_2(x,z)} \delta_z.
\]
This is a renormalized eigenvalue counting measure for a self-adjoint Dirac operator on $[0,x]$ with Dirichlet boundary conditions $f_1(0) = f_2(0)$, $f_1(x) = f_2(x)$. Each measure $\rho_x$ is infinite. The limit of $\rho_x$, if it exists, is a deterministic density of states associated with the operator. The candidate limit is obtained from the Martin function: the Martin function $M_\E$ extends to a subharmonic function on $\bbC$, so it has a Riesz measure $\rho_\E = \frac 1{2\pi} \Delta M_\E$, which we call the Martin measure. Once again, since $\E$ is not semibounded, a modification of the arguments is needed, but we obtain conclusions analogous to the orthogonal polynomial and Schr\"odinger settings:

\begin{theorem} \phantomsection \label{thm114}
\begin{enumerate}[(a)]
\item If $\Lambda_\varphi$ is regular, then $\rho_x\to \rho_\E$ as $x \to \infty$ in the topology dual to $C_c^\infty(\bbR)$.
\item If $\rho_x \to \rho_\E$ as $x \to \infty$ in the topology dual to $C_c^\infty(\bbR)$, then either $\Lambda_\varphi$ is regular or there exists a polar Borel set $B$ such that $\chi_{\bbR \setminus B}(\Lambda_\varphi)= 0$.
\end{enumerate}
\end{theorem}

In the additional case in (b), $\chi_{\bbR \setminus B}(\Lambda_\varphi)$ denotes the spectral projection, in the sense of the Borel functional calculus. This additional case, which prevents Theorem~\ref{thm114} from being an equivalence, is not an artifact of the proof. It stems from the fact that the Riesz measure is uniquely determined by the subharmonic function, but not conversely. In the Jacobi matrix setting, this case is illustrated by the supercritical almost Mathieu operator \cite[Example 8.3]{Simon07}, and we expect that similar examples are possible for Dirac operators.

The results of this paper have immediate corollaries to ergodic Dirac operators \cite{Shubin,Savin,PasturFigotin}. The various averages then exist by ergodicity and are the central objects of ergodic spectral theory: the growth rate of the Dirichlet solution is equal to the Lyapunov exponent except on a polar subset of $\E$, and  the limit of the measures $\rho_x$ is the density of states. Our results thus specialize to the ergodic setting; for instance, by Theorem~\ref{RegularityTFAE}, if an ergodic Dirac operator has Lyapunov exponent $0$ for almost every $z \in \E$ with respect to harmonic measure, then it is regular. We also note 
that Corollary~\ref{corL2precompact} applies to ergodic Dirac operators obtained by continuously sampling on a compact parameter space (e.g. almost periodic Dirac operators). Beyond these remarks, we list explicitly one ergodic application of our results, an ``ultimate Ishii--Pastur theorem" for Dirac operators.

\begin{theorem}\label{IshiiPastur}
Let $\gamma$ denote the Lyapunov exponent associated to an ergodic family of Dirac operators $(\Lambda_{\varphi_\eta})_{\eta \in S}$ which obey \eqref{varphiL2locunif} almost surely. Let $\mu_\eta$ denote a maximal spectral measure for $\Lambda_{\varphi_\eta}$. Let $Q \subset \bbR$ be the Borel set of $\lambda \in \bbR$ with $\gamma(\lambda) > 0$. Then for a.e.\ $\eta \in S$, there exists a polar set $X_\eta$ such that $\mu_\eta(Q \setminus X_\eta) = 0$. In particular, the measure $\chi_Q d\mu_\eta$ is of local Hausdorff dimension zero.
\end{theorem}

We also obtain a Dirac analog of a Widom criterion:

\begin{theorem} \label{theoremWidomcriterion}
Let $\varphi$ obey \eqref{varphiL2locunif} and denote by $\mu$ a maximal spectral measure for $\Lambda_\varphi$. If harmonic measure on $\E$ is absolutely continuous with respect to $\mu$, then $\Lambda_\varphi$ is regular.
\end{theorem}

In Section~\ref{sectionEigensolutions}, we study the asymptotics of eigensolutions and prove a key two-term expansion for Dirichlet eigensolutions. Section~\ref{sectionLimits1} uses this to study subsequential limits of $\frac 1x \log \lVert U(x,z)\rVert$ and, in particular, prove Theorems~\ref{thm11}, \ref{thm12}, and \ref{thm13}. Section~\ref{sectionvague} considers the vague convergence of the measures $\sigma_x$, and Section~\ref{sectionLimits2} combines these techniques to prove all the remaining theorems.

\section{Dirac operators} \label{sectionDiracBackground}

We use the triple norm notation for uniform local $L^p$ bounds: for $p \in [1,\infty)$, we denote
\[
\vertiii{\varphi}_p = \sup_{x \ge 0} \left( \int_x^{x+1} \lvert \varphi(t) \rvert^p \,dt \right)^{1/p}.
\]
In particular, the condition \eqref{varphiL2locunif} is $\vertiii{\varphi}_2 < \infty$. The set of functions which obey $\vertiii{\varphi}_p < \infty$ is customarily denoted $L^p_{\loc,\unif}([0,\infty))$.

In this section we will collect some fundamental properties of Dirac operators in our setting. We assume $\vertiii{\varphi}_1 < \infty$ unless stated otherwise, and always assume that $\varphi$ is defined on the interval appropriate for the current discussion and regular at any finite endpoints.

We write the form \eqref{ZSO} as
\[
\Lambda_\varphi = - i j \partial_x + \Phi, \qquad j = \begin{pmatrix} -1 & 0 \\ 0 & 1 \end{pmatrix}, \qquad \Phi = \begin{pmatrix} 0 & \varphi \\ \bar \varphi & 0 \end{pmatrix}.
\]
We call $V = \binom{v_1}{v_2}$ an eigensolution of $\Lambda_\varphi$ at energy $z$ if $V$ is a locally absolutely continuous function and $V$ satisfies
\begin{equation}
-i j \partial_x V + \Phi V = z V. \label{eigenequation} 
\end{equation}
We define the Wronskian of two eigensolutions $U,V$ as 
 \[
W(U,V) = U^t JV, \qquad J= \begin{pmatrix} 0 & i \\ -i & 0 \end{pmatrix}.
 \]
If $U, V$ are eigensolutions at the same energy $z$, a direct calculation of $\partial_x W(U,V)$ shows that the Wronskian is independent of $x$, and the Wronskian is nonzero if and only if $U,V$ are linearly independent.

Similarly, we define transfer matrices $T(x,z) = T(x,z,\varphi)$ by
\begin{equation}\label{19nov2}
-i j \partial_x T + \Phi T = z T, \qquad T(0,z) = I,
\end{equation}
where the Wronskian being independent of $x$ implies that $\det T(x,z) \equiv I.$ Further, we have that 
\begin{equation}\label{19oct1}
j - T(x,z)^* j T(x,z) = 2 \Im z \int_0^x T(y,z)^* T(y,z) dy.
\end{equation}
For any $z \in \bbC_+$, \eqref{19oct1} is strictly increasing in $x$, so the corresponding Weyl disks
\begin{equation}\label{19nov1}
D(x,z) = \left\{ w \mid \begin{pmatrix} w \\ 1 \end{pmatrix}^* T(x,z)^* j T(x,z)  \begin{pmatrix} w \\ 1 \end{pmatrix} \ge 0 \right\}
\end{equation}
are strictly nested, that is, $D(x_2,z) \subset \intt D(x_1,z)$ if $x_1 < x_2$. It follows from \eqref{19nov1} that  $D(0,z) = \overline{\bbD}$, so our Weyl disks are a nested family of disks in $\overline{\bbD}$. We will need a more uniform statement about the strict nesting:

\begin{lemma}\label{disks.as.fn.of.phi}
For any $R < \infty$, there exists $\delta > 0$ such that if ${\varphi \in L^2([0,1])}$ with ${\int_0^1 \lvert \varphi(y)\rvert^2 dy \le R}$, then ${D(1,z,\varphi) \subset \{ w \mid \lvert w \rvert \le 1 - \delta \}}$.
\end{lemma}

\begin{proof}
Using the series representation of the solution of the initial value problem \eqref{19nov2}, it is proved in \cite[Prop. 1.2]{GreKap} that $T(1,z,\varphi)$ is a compact map of $\varphi \in L^2([0,1])$, in the sense that $\varphi_n \wto \varphi$ implies $T(1,z,\varphi_n) \to T(1,z,\varphi)$. In particular, $T$ maps weakly compact sets to compact sets. Moreover, the maximum distance from the origin to elements of $D(1,z)$ is an explicit function of the matrix entries, since for any disk $D_M$ represented in projective coordinates as $\binom{w_1}{w_2}^*M\binom{w_1}{w_2}$ with $M = \begin{pmatrix} m_1 & m_2 \\ m_3 & m_4 \end{pmatrix}$ and $m_1 < 0$, $M^* = M$, $\det M < 0$,  $D_M$ is a proper disk in $\C$. Its radius $r$ and center $c$ are given by
\[
	r = -\frac{\sqrt{-\det M}}{m_1}, \qquad  c = -\frac{m_2}{m_1}.
\]
In this notation, $D(1,z) = D_{T(1,z)^*jT(1,z)}$, so calling $T(1,z) = \begin{pmatrix} t_1 & t_2 \\ t_3 & t_4 \end{pmatrix},$ we find that the radius and center of $D(1,z)$ are	
\[
	r = \frac{1}{|t_1|^2 - |t_3|^2} , \qquad c = \frac{\overline{t_3}t_4 - \overline{t_1}t_2}{|t_1|^2 - |t_3|^2}.
\]
Here it is important to note that the identity \eqref{19oct1} implies $j - T^* j T > 0$ and, by computing its top left entry, that $|t_1|^2 - |t_3|^2 > 1$. Thus, the maximum distance from the origin to elements of $D(1,z)$ is given by
\[
\rho(z,\varphi) := \max \{ \lvert w \rvert \mid w \in D(1,z,\varphi) \} = \frac{1+ |\overline{t_3}t_4 - \overline{t_1}t_2|}{|t_1|^2-|t_3|^2},
\]
which depends continuously on the entries of $T$. Since $\rho(z,\varphi) < 1$ for any $\varphi \in L^2([0,1])$, by compactness $\rho$ achieves a maximum $1-\delta < 1$ on the weakly compact set ${\{\varphi \in L^2([0,1]): \|\varphi\|_2 \leq R\}.}$   
\end{proof}

When $\vertiii{\varphi}_1 < \infty$, the operator is in the limit point case at $\infty$, i.e., the Weyl disks shrink to a point for every $z \in \bbC_+$. The intersections
\[
\{ s(z,\varphi) \} = \{ s(z) \} = \bigcap_{x \in (0,\infty)} D(x,z)
\]
define the spectral function $s(z)$ of the Dirac operator. This is a Schur function in the sense that it is an analytic function from $\bbC_+$ to $\bbD,$ and so we call $s(z)$ the \textit{Schur function for $\Lambda_\varphi$}. The same class of Schur functions appears in the study of canonical systems with signature matrix $j$ \cite{BLY1}.

\begin{remark}
The forms \eqref{ZSO}, \eqref{DO} are related by  $L_\varphi =  P^{-1} \Lambda_\varphi P$, where $P:= \begin{pmatrix} 1 & i \\ 1 & - i \end{pmatrix}$. Whereas we have noted that the Weyl disk in the Zakharov-Shabat gauge, $D(0,z)=\overline{\D},$ is defined by the inequality $\binom{w_1}{w_2}^*j\binom{w_1}{w_2} \geq 0,$ the correct definition of the Weyl disk in the Dirac gauge is defined by $\binom{w_1}{w_2}^* P\inv j P \binom{w_1}{w_2} \geq 0,$ or equivalently $2\Im (\overline{w_1}w_2) \geq 0,$ which places $w_2/w_1$ in $\C_+$. This explains why the Weyl disks for $\Lambda_\varphi$ are in $\bbD$ whereas those for $L_\varphi$ are in $\bbC_+$. This seemingly cosmetic difference makes the form $\Lambda_\varphi$ more convenient, and works based on the form $L_\varphi$ often pass to the form $\Lambda_\varphi$ in technical parts of the arguments \cite{ClaGes02}. The Weyl functions for $\Lambda_\varphi$ and $L_\varphi$ are related by a Cayley-type transform,
\begin{equation}\label{eqnCayley}
s(z) = \frac{1 + im(z)}{1 - im(z)} = -\frac{m(z) - i}{m(z) + i}.
\end{equation}
Similarly, self-adjoint boundary conditions for $\Lambda_\varphi$ correspond to $\gamma\in \bbC^2$ such that $\gamma^* j \gamma = 0$ so they are parametrized by the unit circle. In particular, the condition $f_1(0) = f_2(0)$ in \eqref{15jun1} corresponds to a Dirichlet boundary condition at $0$ for the corresponding operator $L_\varphi$, and accordingly we call an eigensolution of \eqref{eigenequation} obeying this boundary condition a \textit{Dirichlet solution}.
\end{remark}

In the limit point case, for any $z \in \bbC \setminus \bbR$, the \textit{Weyl solution} can be defined as a nontrivial eigensolution $\Psi$ of $\Lambda_\varphi$ such that
\[
s(z) = \frac{\psi_1(0,z)}{ \psi_2(0,z)}.
\]
This determines the Weyl solution uniquely up to normalization, and a calculation similar to \eqref{19oct1} gives $\Psi \in L^2([0,\infty), \C^2)$.  More generally, we can define
\[
s(x,z) = \frac{\psi_1(x,z)}{\psi_2(x,z)}
\]
and the definition of $D(x,z)$ implies that $s(x,z) \in \overline{\D}$ for all $x \in [0,\infty).$ In particular, $\psi_2(x,z)$ is the dominant component of $\Psi(x,z)$ for all $x \in [0,\infty).$ Note that, in contrast with the Schr\"odinger case, the Schur function is not the logarithmic derivative of the Weyl solution; this will be significant in proofs which capture the decay rate of the Weyl solution by taking the average of its logarithmic derivative.

From the identity \eqref{19oct1}, we see that for $c > x$ the condition
\begin{equation}
	U(c,z)^*j U(c,z) \geq 0 \label{weyl.disk.condition}
\end{equation} 
implies that $U(x,z)^* j U(x,z) > 0$, i.e., $s_U(x,z) := \frac{u_1(x,z)}{u_2(x,z)} \in \D.$ Thus for any solution $U$ obeying \eqref{weyl.disk.condition}, we call $s_U$ the \textit{Schur function for $U$}. With this definition, the Schur function, without decoration, is the Schur function for the Weyl solution $\Psi$. This definition allows an alternative characterization of $D(x,z)$:
\[
	D(x,z) = \left\{s_U(0,z): \Lambda_\varphi U = zU \text{ and } U(x,z)^*j U(x,z) \geq 0\right\}
\]
The radius of $D(x,z)$ decays exponentially as $x \to \infty$--to see this, one can use the above description of $D(x,z)$ together with the following differential equation that $s_U$ obeys: 

\begin{lemma}
For $V(x,z)$ an eigensolution of \eqref{eigenequation}, the Schur function for $V$, $s_V(x,z),$ obeys the following Riccati-type differential equation:
\begin{equation} \label{schur.riccati}
	- i \partial_x s_V(x, z) = \overline{\varphi(x)} s_V(x, z)^2 - 2 z s_V(x, z) + \varphi(x).
\end{equation}
\end{lemma}

\begin{proof}
Rewriting $V$ as $V(x,z) = v_2(x,z) \binom{ s_V(x,z) }{ 1}$, and using the fact that $V$ solves \eqref{eigenequation} gives
\[
\partial_x v_2(x,z) \binom{ i s_V(x,z)}{-i} + v_2(x,z) \binom{ i \partial_x s_V(x,z) }{0} + v_2(x,z) \binom{ \varphi(x) }{\overline{\varphi(x)} s_V(x,z)} = z v_2(x,z) \binom{ s_V(x,z) }{1}.
\]
Left multiplication by $\begin{pmatrix} 1 & s_V(x,z) \end{pmatrix}$ leads to \eqref{schur.riccati}.
\end{proof}

\begin{lemma}\label{schur.fxns.close}
If $U(x,z), V(x,z)$ solve \eqref{eigenequation} for $x \in [a,b]$ and satisfy \eqref{weyl.disk.condition} at $b$, their Schur functions $s_U(x,z)$ and $s_V(x,z)$ obey  
\[
	|s_U(a, z) - s_V(a, z)| \leq 2 \exp\Big(- 2 \Im (z) (b - a) + 2 \int_{a}^{b} |\varphi(t)| dt\Big).
\] 
In particular, the radius of $D(x,z)$ decays exponentially as $x \to \infty$ for $\Im z > \vertiii{\varphi}_1$.
\end{lemma}
\begin{proof}
For $\Gamma(x, z) := i(\overline{\varphi(x)} (s_U(x, z) + s_V(x, z)) - 2z)$ and for $x \in [a,b]$, we use \eqref{schur.riccati} to obtain 
\[
	\partial_x(s_U(x, z) - s_V(x, z)) = \Gamma(x, z) (s_U(x, z) - s_V(x, z)).
\]
A solution $y(x, z) = s_U(x, z) - s_V(x, z)$ can be written as
\[
	y(x, z) = y(b,z) e^{2 i z (b - x)} - \int_{x}^{b} e^{2 i z (t - x)} (\Gamma(t, z) + 2 i z) y(t, z) dt.
\]
A Volterra-type argument yields the series expansion for $y$,
\[
	y(x, z) = y(b,z) e^{2 i z (b - x)} \left(1 + \sum_{n = 1}^{\infty} (-1)^n \int_{x}^{b} \int_{t_1}^{b} \cdots \int_{t_{n - 1}}^{b} \prod_{k = 1}^{n} (\Gamma (t_k, z) + 2 i z) d^n t\right).
\]
By the definition of $\Gamma(x, z),$ we have that $\Gamma(t, z) + 2 i z = i \overline{\varphi(t)}(s_U(t, z) + s_V(t, z))$. The condition \eqref{weyl.disk.condition} implies that $|s_U(x, z)|, |s_V(x,z)| \leq 1$ for all $x\in [a,b]$ and $z \in \bbC_+$. Evaluating $y$ at $x=a$, taking the modulus, and bounding $|s_U+s_V|$ and $|y(b,z)|$ by $2$, we realize the series as an exponential to prove the claimed bound. The statement about $D(x,z)$ follows by taking $a=0$ and $b=x.$ 
\end{proof}
We mentioned above that the logarithmic derivative of the Weyl solution is not the Schur function. However, with $s(x,z)$ the Schur function for the Weyl solution, from the second row of \eqref{eigenequation} we have 
\[
	\frac{ \partial_x \psi_2(x,z) }{\psi_2(x,z)} = i z - i \overline{\varphi(x)} s(x,z).
\]
This identity in fact holds replacing $\psi_2$ with $u_2$ and $s(x,z)$ with $s_U(x,z)$ for any solution $U$ obeying \eqref{weyl.disk.condition}. Consequently, this family of identities together with Lemma \ref{schur.fxns.close} provides an estimate on the difference of logarithmic derivatives of solutions obeying \eqref{weyl.disk.condition}. In particular, the Dirichlet solution will be accessible via the following reflection symmetry for Dirac operators: 

\begin{lemma}\label{reflection_symmetry}
For $-\infty\leq a<b\leq +\infty$, if $V^+$ is an eigensolution for \eqref{eigenequation} with operator data $\varphi$ on $(a,b)$ and $\tilde \varphi(x) := \overline{\varphi(-x)}$, then 
\[
	V^-(x,z;\tilde \varphi) :=\begin{pmatrix} 0 & 1 \\ 1 & 0 \end{pmatrix} V^+(-x,z;\varphi),
\]
is an eigensolution for $\eqref{eigenequation}$ with operator data $\tilde{\varphi}$ on $(-b, -a)$. 
\end{lemma}

\begin{proof}
Applying $\Lambda_{\tilde{\varphi}}$ to $V^-(x,z;\tilde \varphi)$ yields
\begin{align*}
\left[-ij \partial_x + \begin{pmatrix} 0 & \overline{\varphi(-x)} \\ \varphi(-x) & 0 \end{pmatrix}\right] \begin{pmatrix} 0 & 1 \\1 & 0 \end{pmatrix} V^+(-x,z) &= \begin{pmatrix} 0 & 1 \\1 & 0 \end{pmatrix} \left[ij \partial_x + \begin{pmatrix} 0 & \varphi(-x) \\ \overline{\varphi(-x)} & 0 \end{pmatrix}\right]V^+(-x,z).
\end{align*}
Making the change of variable $t = -x$, since $V^+$ solves \eqref{eigenequation} with operator data $\varphi$ we have
\[
	\begin{pmatrix} 0 & 1 \\1 & 0 \end{pmatrix} \left[-ij \partial_t + \begin{pmatrix} 0 & \varphi(t) \\ \overline{\varphi(t)} & 0 \end{pmatrix}\right]V^+(t,z) = z \begin{pmatrix} 0 & 1 \\1 & 0 \end{pmatrix} V^+(t,z) = z V^-(x,z),
\]
which proves the claim.
\end{proof}

If we associate with the Dirichlet solution $U^+$ the eigensolution, $U^-,$ of \eqref{eigenequation} on the interval $(-\infty,0]$ with operator data $\tilde{\varphi}$ as in the lemma, then $U^-$ obeys \eqref{weyl.disk.condition}. Thus $u_1^-(x,z)/u_2^-(x,z) \in \D$ for $x<0$. Moreover, in order to make use of Lemma \ref{disks.as.fn.of.phi}, we can extend our definition of the Weyl disks to allow the case when $x<0$ as follows:
\[
	D^-(x,z,\varphi) := \left\{\frac{u_2(0,z)}{u_1(0,z)} \mid \Lambda_{\tilde{\varphi}}U=zU \text{ and } U(x,z)^* j U(x,z) \leq 0\right\} \text{ for } x<0.
\]
This definition is compatible with the reflection symmetry described in Lemma~\ref{reflection_symmetry}, which switches between Weyl disks $D^+$ and $D^-$; in particular, $D^-(x,z;\varphi)$ inherit the same estimates, like those in Lemma~\ref{disks.as.fn.of.phi}. Combining this reflection symmetry with the translations $\varphi_x(t) = \varphi(t+x)$, we will have robust uniform estimates at our disposal.  We use this to compare the rates of growth and decay of the Dirichlet and Weyl solutions:

\begin{lemma}
\label{u.times.psi.inequality}
For each $z \in \C_+$, the Weyl and Dirichlet solutions obey
\[
	C\inv < |u_1(x,z) \psi_2(x,z)| < C
\]
for a finite constant $C >1$ independent of $x \in [1,\infty)$. 
\end{lemma} 
\begin{proof}
For a fixed $z \in \C_+$, $(u_2/u_1)(x,z)$ and $(\psi_1/\psi_2)(x,z)$ are Schur functions for any $x>0$. Thus, $|1- \frac{u_2 \psi_1}{u_1 \psi_2}(x,z)| <2$ and so 
\[
	|u_1(x,z)\psi_2(x,z)| > \frac{1}{2}|u_1(x,z)\psi_2(x,z)| \left\lvert 1- \frac{u_2 \psi_1}{u_1 \psi_2}(x,z) \right\rvert = \frac{1}{2}|W(U,\Psi)|,
\]
where the Wronskian $W$ is $x$-independent and nonzero since $U$ and $\Psi$ are linearly independent solutions. The argument for an upper bound is similar once we have a lower bound on $|1- \frac{u_2 \psi_1}{u_1 \psi_2}|$. By Lemma \ref{disks.as.fn.of.phi} there is a $\delta >0$ such that for every ${\|\tilde{\varphi}\|_{L^2([0,1])} \leq \vertiii{\varphi}_2,}$ ${D(1,z,\tilde{\varphi}) \subset \{w\mid  \vert w \vert \leq 1-\delta\}}$. In particular, for every $x$, $D(1,z,\varphi_x)$ and $D^-(1,z,\varphi_x)$ are both subsets of ${\{w\mid \vert w\vert \leq 1-\delta\}}$. Since the functions $(\psi_1/\psi_2)(x,z)$ and $(u_2/u_1)(x,z)$ are elements of $D(1,z, \varphi_x)$ and $D^-(1,z,\varphi_x),$ respectively, their product is bounded uniformly in $x\in [1,\infty)$ by $1-\delta$. We conclude
\[
	|u_1(x,z)\psi_2(x,z)| = |u_1(x,z)\psi_2(x,z)| \left|\frac{1- \frac{u_2 \psi_1}{u_1 \psi_2}}{1- \frac{u_2 \psi_1}{u_1 \psi_2}}\right| < \frac{1}{\delta} |W(U,\Psi)|,
\]
so that taking $C = \frac{1}{\delta} |W(U,\Psi)|$ gives the lemma.
\end{proof}
 
We will need a function symmetric in conjugation of $z$, that is, a function obeying
\begin{equation}
	h(\overline{z}) = h(z), \label{hsymmetry}
\end{equation} 
which shares the asymptotic behavior of $|u_1(x,z)|$. Properties of the equation \eqref{eigenequation} and the above reflection symmetry for solutions imply that $h(x,z) = \frac{1}{x} \log |u_1(x,z) - u_2(x,z)|$ exhibits this symmetry, and in order to study the asymptotic behavior of $h$ via the asymptotic behavior of $\frac{1}{x} \log \lvert u_1(x,z)\rvert$, we will need the following lemma: 

\begin{lemma}\label{pass.to.new.h}
For each $z \in \C_+$, there exists a finite positive constant $C$ such that
\[
	C\inv \leq \frac{|u_1(x,z) - u_2(x,z)|}{|u_1(x,z)|} \leq C
\]
\end{lemma}
\begin{proof}
The triangle inequality yields 
\[
	\left\lvert 1 - \frac{|u_2(x,z)|}{|u_1(x,z)|} \right\rvert \leq \frac{|u_1(x,z) - u_2(x,z)|}{|u_1(x,z)|} \leq 1 + \frac{|u_2(x,z)|}{|u_1(x,z)|}.
\]
The left hand side can be bounded below by the same $\delta$ as in the proof of Lemma \ref{u.times.psi.inequality} and the right hand side can be bounded above by two, since $\frac{u_2(x,z)}{u_1(x,z)}$ is a Schur function.
\end{proof}

We will also need a Schnol-type result:

\begin{lemma}\label{lemmaSchnol}
Let $\Lambda_\varphi$ be a half-line Dirac operator such that  \eqref{varphiL2locunif} holds. Let $\mu$ denote its maximal spectral measure. Fix $\kappa > 1/2$. For $\mu$-a.e.\ $z$, $\lVert U(x,z) \rVert = O(x^\kappa)$ as $x \to \infty$.
\end{lemma}

\begin{proof}
Fix $z \in \bbC \setminus \bbR$. The resolvent of a Dirac operator is an integral operator,
\[
((\Lambda_\varphi -z)^{-1} f)(x) = \int G(x,y;z) f(y) dy.
\]
Note that $G(x,y;z)$ is matrix-valued.

Denote  $e_1 =\binom 10$, $e_2 = \binom 01$.  By \eqref{15jun1}, $D(\Lambda_\varphi) \subset L^\infty([0,\infty),\bbC^2)$. Thus, for any $f \in L^2([0,\infty), \bbC^2)$,
\[
\sup_{x\in [0,\infty)} \left\lvert \int e_j^* G(x,y;z) f(y) dy \right\rvert < \infty,
\]
so by the uniform boundedness principle,
\begin{equation}\label{25nov1}
\sup_{x \in [0,\infty)} \int \lVert e_j^* G(x,y;z) \rVert^2 \,dy < \infty.
\end{equation}
The eigenfunction expansion for the Dirac operator is a unitary map $\cU : L^2([0,\infty),\bbC^2) \to L^2(\bbR,d\mu)$ with $\mu$ the canonical maximal spectral measure, which conjugates $\Lambda_\varphi$ to the operator given by multiplication by $\lambda$ on $L^2(\bbR,d\mu(\lambda))$. This map is given  on compactly supported functions by ${(\cU f)(x) = \int U(x,\lambda)^* f(x) dx}$ and extended by continuity. For fixed $x > 0$, the functions $f_j(y) = G(x,y;z) e_j$ are mapped to $(\cU f_j)(\lambda) = \frac 1{\lambda - z} U(x,\lambda)^* e_j$ (formally, this follows from $\delta_x e_j \mapsto U(x,\lambda)^* e_j$ and the fact that $\cU$ conjugates $(\Lambda_\varphi-z)^{-1}$ to multiplication by $(\lambda -z)^{-1}$). In particular, since $\cU$ is unitary, computing squares of $L^2$ norms gives
\begin{equation}\label{25nov2}
 \int \lVert G(x,y;z) e_j \rVert^2 \,dy = \int  \frac {\lVert U(x,\lambda) e_j \rVert^2}{\lvert \lambda - z \rvert^2} d\mu(\lambda). 
\end{equation}
Summing \eqref{25nov1}, \eqref{25nov2} over $j=1,2$ gives statements for the Hilbert--Schmidt norm of $G(x,y;z)$; combining them gives
\[
\sup_{x\in [0,\infty)} \int  \frac {\lVert U(x,\lambda) \rVert^2}{\lvert \lambda - z \rvert^2} d\mu(\lambda) < \infty.
\]
Multiplying by the integrable function $(1+x^2)^{-\kappa}$ and integrating in $x$ gives
\[
\iint (1+x^2)^{-\kappa}  \frac {\lVert U(x,\lambda) \rVert^2}{\lvert \lambda - z \rvert^2} d\mu(\lambda) dx < \infty.
\]
Thus, by Fubini's theorem,
\begin{equation}\label{25nov4}
\int (1+x^2)^{-\kappa} \lVert U(x,\lambda) \rVert^2 dx < \infty, \qquad \mu\text{-a.e. }\lambda.
\end{equation}
Gronwall's inequality for $\Lambda_\varphi U = zU$ gives an estimate on $\lVert U(x,\lambda) \rVert \le C \int_x^{x+1} \lVert U(y,\lambda) \rVert dy$ with an $x$-independent value of $C$. Thus, \eqref{25nov4} implies $\lVert U(x,\lambda)\rVert = O(x^\kappa)$ for $\mu$-a.e.\ $\lambda$.
\end{proof}

\section{The Martin function}\label{sectionMartin}

In this section we will recall some facts about the Martin boundary and discuss in more detail the Akhiezer-Levin condition for unbounded closed subsets of $\R$. The semibounded case was considered in \cite{EL}; in the following let us assume that $\E \subset \bbR$ is a closed subset which is unbounded from below and from above. For more details on the potential theoretical concepts we will introduce below we refer to \cite{ArmGar01,RansPotential}. Let $\Omega=\C\setminus \E$, $z_0\in\Omega$ and denote the Green function of $\Omega$ with pole at $z_0$ by $G_\E(z,z_0)$. The Martin kernel normalized at $z_*\in\Omega$ is defined on $\Omega\times(\Omega\setminus\{z_*\})$ by 
\begin{align*}
M_\E(z,z_0)=\frac{G_\E(z,z_0)}{G_\E(z_*,z_0)}.
\end{align*}
The Martin compactification $\hat{\Omega}$ is the smallest metric compactification of $\Omega$ such that $M_\E(z,\cdot)$ can be continuously extended to the boundary $\partial^M\Omega=\hat{\Omega}\setminus\Omega$. We will also write $M_\E(z,z_0)$ for the extended function. Note that by the Harnack principle $\{M_\E(z,z_0)\}_{z_0\in\Omega}$ is precompact in the space of positive harmonic functions equipped with uniform convergence on compacts and thus for any $z_0\in \hat{\Omega}$, $M(\cdot,z_0)$ defines a positive harmonic function in $\Omega$. For $M$ a positive harmonic function in $\Omega$, we call $M$ minimal if for every positive harmonic function, $h$, with $0\leq h\leq M$, we can conclude that $h=cM$ for some $c\geq 0$. We define $\partial_1^M\Omega\subset\partial^M\Omega$ as the subset of the Martin boundary which consists of minimal harmonic functions. It is a general result from Martin theory that every positive harmonic function admits an integral representation in terms of the Martin kernel, that is, there exists a measure $\nu$ on $\partial_1^M\Omega$ such that 
\begin{align}\label{eq:MartinRepPosHarm}
h(z)=\int_{\partial_1^M\Omega}M_\E(z,x)d\nu(x),\quad h(z_*)=\nu(\partial_1^M\Omega).
\end{align}
The Martin boundary for Denjoy domains is quite intuitive. For any Euclidean point $x\in\E\cup\{\infty\}$ there are either one or two minimal Martin boundary points associated to it, depending how ``dense'' the set $\E$ is locally at this point. Roughly speaking, if $\E$ is sufficiently dense at $x$, then locally $\Omega$ splits into the two half planes $\bbC_+$ and $\bbC_-$ and we obtain a minimal Martin function for both of them. That is, in \cite[Theorem 6]{GardSjoed09} it is shown that there exists a map $\pi: \partial^M_1\Omega\to \E\cup \{\infty\}$ such that for every $x\in\E\cup\{\infty\}$,  $\#\pi^{-1}(\{x\})$ is either one or two.
The precise statement requires the concept of \textit{minimal thinness}.  If $A$ is a subset of the Martin boundary $\partial^M\Omega=\hat \Omega\setminus\Omega$, then we say a property, $P$, holds near $A$ if there is a Martin-neighborhood $A\subset W$ such that $P$ holds on $W\cap \Omega$. Then, for  $A\subset \hat\Omega$ and a positive superharmonic function $h$ on $\Omega$ we define the reduced function
\begin{align}\label{def:reduction}
	R_h^A(x)=\inf\{u(x):\ u\geq 0\text{ is superharmonic, } h\leq u \text{ on } A\cap \Omega\text{ and }h\leq u \text{ near } A\cap \partial^M\Omega\}
\end{align}
and $\hat R_h^A$ denotes its lower semicontinuous regularization. A set $A\subset \Omega$ is said to be minimally thin at $y\in \partial^M_1\Omega$ if 
\begin{align*}
	\hat R^A_{M_\E(\cdot,y)}\neq  M_\E(\cdot,y).
\end{align*}
Then $\#\pi^{-1}(\{x\})=2$ if and only if there is $y\in\pi^{-1}(\{x\})$  such that $\Omega\cap \R$ is minimally thin at $y$. 
For more information we refer the reader to \cite[Chapter 9.2]{ArmGar01} and to \cite{GardSjoed09} for a survey on Martin theory for Denjoy domains. 

We are particularly interested in the case $x=\infty$. In this case, we define 
\begin{align*}
\cM_\E(\infty)=\bigcap_{K>0}\overline{\{M_\E(z,z_0):\ z_0\in\Omega, |z_0|>K\}}^M
\end{align*}
as the set of all Martin functions associated to $\infty$. Let us also denote by $M_\E(z,\infty_{\pm})$ the possible two minimal Martin functions at $\infty$. It is well known that 
\begin{align}\label{eq:MartinFunctionSet}
\cM_\E(\infty)=\{\l_+M_\E(z,\infty_{+})+\l_-M_\E(z,\infty_{-}):\ \l_\pm \geq 0, \l_++\l_-=1\}.
\end{align} 
 In any case, 
 \begin{align*}
 M_\infty(z)=\frac{1}{2}(M_\E(z,\infty_{+})+M_\E(\overline{z},\infty_{+})),
 \end{align*}
 is the unique symmetric function, i.e.,  $M_\infty(z)= M_\infty(\overline{z})$ in $\cM_\E(\infty)$. 
 Define also the cone 
\begin{align*}
\cP_\E(\infty)=\bigcup_{\l> 0}\l\cM_\E(\infty).
\end{align*}
The following theorem presents a list of equivalent characterizations of $\cP_\E(\infty)$. We say that $h$ vanishes continuously at a point $x \in \E$ if $\lim_{\substack{z\to x\\ z \in \Omega}} h(z) = 0$. We call a subset of $\Omega$ bounded if it is bounded as a subset of $\bbC$.
\begin{theorem}\label{lem:ConePositive}
	Let $H_{+,b}(\Omega)$ denote the set of positive harmonic functions on $\Omega$ that are bounded on every bounded subset of $\Omega$. Then, the following are equivalent:
	\begin{enumerate}[(i)]
		\item $h\in H_{+,b}(\Omega)$ and $h$  vanishes continuously for every Dirichlet-regular point of $\E$; 
		\item $h\in H_{+,b}(\Omega)$ and $h$   vanishes continuously quasi-everywhere on $\E$;
		\item $h\in H_{+,b}(\Omega)$ and $h$   vanishes continuously $\omega_\E(\cdot,z_0)$-a.e.;
		\item $h\in \cP_\E(\infty)$.
	\end{enumerate}
\end{theorem}
\begin{proof}
Due to \cite[Remark 5, Theorem 6]{GardSjoed09} $(iv)\implies(i)$. Kellogg's theorem \cite[Corollary 6.4]{GarHarmonicMeasure} yields $(i)\implies(ii)$ and by \cite[Theorem III.8.2]{GarHarmonicMeasure} we get that $(ii)\implies(iii)$. It remains to show that $(iii)\implies(iv)$.
Using \eqref{eq:MartinFunctionSet}, it remains to show that for any $h\in H_{+,b}(\Omega)$ that vanishes continuously $\omega_\E(\cdot,z_0)$-a.e., there are $\l_1,\l_2\geq 0$, so that 
\begin{align*}
h=\l_1M_\E(z,\infty_{+})+\l_2M_\E(z,\infty_{-}).
\end{align*}
 Let $E_n=\{\pi^{-1}(\{x\}): x\in\E,\ |x|<n\}$ and $E_n^\mathsf{c}=\partial_1^M\Omega\setminus E_n$. Let $R_h^{E_n}$ denote the reduction of $h$ (which is harmonic since $E_n\subset \partial_1^M\Omega$).  Since $h\in H_{+,b}(\Omega)$, $h$ is majorized by a constant in $U \cap \Omega$ where $U$ is a neighborhood of $E_n$ in $\hat\Omega$, so $R_h^{E_n}$ is a bounded harmonic function in $\Omega$ which vanishes $\omega_\E(\cdot,z_0)$-a.e. on the boundary. By the maximum principle \cite[Theorem 8.1]{GarHarmonicMeasure} it follows that $R_h^{E_n}=0$. Therefore,
 \begin{align*}
 h=R^{\partial_1^M\Omega}_h\leq R_h^{E_n}+R_h^{E_n^\mathsf{c}}=R_h^{E_n^\mathsf{c}}\leq h,
 \end{align*}
 where we used \cite[Lemma 8.2.2, Corollary 8.3.4]{ArmGar01}. That is, $h=R_h^{E_n^\mathsf{c}}$ for all $n$.  We want to show that $\lim_{n\to\infty}R_h^{E_n^\mathsf{c}}=R^{\pi^{-1}(\{\infty\})}_h$. Since $\pi^{-1}(\{\infty\})\subset E_n^c$, it follows by the definition of the reduction operator that $R^{\pi^{-1}(\{\infty\})}_h\leq R_h^{E_n^\mathsf{c}}$. Since for any open neighborhood, $W$, of $\pi^{-1}(\{\infty\})$ in $\hat\Omega$, there exists $n$ such that $E_n^\mathsf{c}\subset W$, we obtain from \eqref{def:reduction} that $R_h^{\pi^{-1}(\{\infty\})}=\lim_{n\to\infty}R_h^{E_n^\mathsf{c}}=h$. Thus,
\begin{align*}
h=R_h^{\pi^{-1}(\{\infty\})}.
\end{align*}
It follows from \cite[Lemma 8.2.11]{ArmGar01} that there exists a measure $\nu$ supported on $\pi^{-1}(\{\infty\})$ so that 
\begin{align*}
R_h^{\pi^{-1}(\{\infty\})}(z)=\int M(z,x)d\nu(x)=\nu_+M_\E(z,\infty_+)+\nu_-M_\E(z,\infty_-),
\end{align*}
which finishes the proof. 
\end{proof}

The symmetric Martin function $M_\infty(z)$ defines a positive harmonic function on $\bbC_+$ and therefore admits an integral representation 
\begin{equation}\label{22nov1}
M_\infty(x+iy)=a_\infty y+\int\frac{y}{(t-x)^2+y^2}d\nu_\infty(t),\quad \int\frac{d\nu_\infty(t)}{1+t^2}<\infty
\end{equation}
and 
\begin{align}\label{eq:posharmGrowth}
0\leq a_\infty=\lim\limits_{y\to\infty}\frac{M_\infty(iy)}{y}.
\end{align}
We call $\E$ an \textit{Akhiezer--Levin set}, if $\a_\infty>0$, that is, if $M_\infty$ has the maximal possible growth as $z\to i\infty$. It is shown in \cite[Section 3.1]{Lev89Par3} that $\E$ is an Akhiezer--Levin set if and only if $\cP_\E(\infty)$ is two-dimensional. In this case we define $M_\E(z)$ to be the symmetric Martin function in $\cP_\E(\infty)$ which is normalized so that \eqref{MartinFuncNormalization} holds. Since $M_\E$ vanishes q.e.\ on $\E$, by \cite[Theorem 5.2.1]{ArmGar01} it can be extended to a subharmonic function to all of $\bbC$. For a subharmonic function $u$ in $\bbC$ define 
\[
B(u,r)=\sup_{|z|=r}|u(z)|,
\]
and define the order of $u$ by 
\begin{align*}
	\sigma=\limsup_{r\to\infty}\frac{\log B(u,r)}{\log r}.
\end{align*}
If $u$ is of order $\sigma$ we call 
\begin{align*}
	\tau = \limsup_{r\to\infty}\frac{B(u,r)}{r^{\sigma}}\in[0,\infty]
\end{align*}
the type of $u$. We say $u$ is of minimal, mean or maximal type, if $\tau$ is zero, positive and finite or infinite, respectively. 
\begin{lemma}\label{lem:positveSymSubharmonic}
 If $u$ is a symmetric, subharmonic function on $\bbC$ so that $u$ is positive and harmonic on $\bbC\setminus\bbR$, then it is at most of order one. If the order is one, it is not of maximal type. Moreover, there are constants $c_1, c_2\in\bbR$, so that 
\begin{align}\label{eq:HadamardRepSub}
u(z)=\int_{(-1,1)} \log|t-x|d\nu(t)+\int_{\bbR \setminus (-1,1)}  \left(\log\left|1-\frac{z}{t}\right|+\frac{\Re z}{t}\right)d\nu(t)+c_1+c_2\Re z
\end{align}
and $\nu$ is the Riesz measure of $u$.
\end{lemma}

\begin{proof}
That $u$ is of order at most one and that it is not of maximal type is proved in \cite[Theorem 1.4]{Lev89}. It then follows from  \cite[Theorem 4.2]{Hayman1}, that $u$ admits the representation 
\begin{align}
u(z)=\int_{(-1,1)} \log|t-x|d\nu(t)+\int_{\bbR \setminus (-1,1)}  \left(\log\left|1-\frac{z}{t}\right|+\frac{\Re z}{t}\right)d\nu(t)+c_1+c_2\Re z+c_3\Im z.
\end{align}
Since $u$ is symmetric, we see that $c_3=0$ and we obtain \eqref{eq:HadamardRepSub}. 
\end{proof}

Thus, in particular the symmetric Martin function can be represented as
\begin{align}\label{eq:HadamardMartin}
	M_\E(z)=\int_{(-1,1)} \log|t-x|d\rho_\E(t)+\int_{\bbR \setminus (-1,1)}  \left(\log\left|1-\frac{z}{t}\right|+\frac{\Re z}{t}\right)d\rho_\E(t)+c_1+c_2\Re z.
\end{align}

We emphasize that the limit \eqref{eq:posharmGrowth} exists in $[0,\infty)$ for any positive harmonic function $h$. In view of \eqref{eq:MartinRepPosHarm} this growth should also be reflected in the growth rate of $M_\infty$, leading to the following criterion for $\E$ to be an Akhiezer--Levin set.
\begin{lemma}
\label{AL.criterion}
Assume there is a positive harmonic function in $\Omega$ with the symmetry \eqref{hsymmetry} such that 
\begin{equation}\label{12nov1}
\lim\limits_{y\to\infty}\frac{h(iy)}{y}=1.
\end{equation}
Then $\Omega$ is Greenian, $\infty$ is a Dirichlet regular point, $\E$ is an Akhiezer--Levin set, and 
\begin{equation}\label{22nov7}
M_\E(z)\leq h(z), \qquad \forall z \in \Omega.
\end{equation}
\end{lemma}

\begin{proof}
Since $h$ is nonconstant, $\Omega$ is Greenian. Moreover, by a general principle,
\begin{equation}\label{12nov2}
\lim_{\substack{z \to \infty \\ \arg z \in [\delta,\pi -\delta]}} \frac{h(z)}{M_\infty(z)} = \inf_{z \in \Omega} \frac{h(z)}{M_\infty(z)}.
\end{equation}
In particular, the limit exists and is finite. This was proved by Borichev--Sodin \cite{BS01} in the Dirichlet-regular sets $\E$ but it holds without that assumption; see \cite[Cor. 4.20]{BLY2}.

If $h$ obeys \eqref{12nov1}, finiteness of the limit \eqref{12nov2} implies that $\E$ is an Akhiezer--Levin set; then $M_\infty$ can be replaced by $M_\E$ in \eqref{12nov2}, and the normalization \eqref{MartinFuncNormalization}  concludes the proof.
\end{proof}

Recall that $\infty$ being Dirichlet regular is equivalent to the fact that $G_\E(z,z_0)$ vanishes continuously at $\infty$ for some and hence all $z_0\in\Omega$. Carleson and Totik \cite{CarTot04} have shown that $\E$ being an Akhiezer--Levin set implies that $G_\E(z,z_0)$ is even Lipschitz continuous at $\infty$. 

In the Akhiezer--Levin case, upper bounds on the Martin function obtained from \eqref{22nov7} will be combined with the universal lower bound
\[
M_\E(z) \ge \lvert \Im z \rvert, \qquad \forall z \in \bbC \setminus \bbR.
\]
This bound follows from \eqref{22nov1}, because \eqref{eq:posharmGrowth} gives $a_\infty=1$. 

Lemma~\ref{AL.criterion} will be used in tandem with statements based on $\bbC_+$:

\begin{lemma}
\label{expansion.for.linear.h}
Assume that $h$ is a positive harmonic function  on $\bbC_+$ such that
\begin{equation}\label{19nov4}
h(iy) =  y + O(1/y), \qquad y \to \infty.
\end{equation}
Then there exists a constant $a \ge 0$ such that 
\begin{equation}\label{19nov5}
h(z) = \Im \Big(z - \frac az\Big) + o(|z|\inv)
\end{equation}
as $z \to \infty$, $\delta \leq \arg z \leq \pi - \delta$ for any $\delta >0$.
\end{lemma}
\begin{proof}
Since $h$ is positive harmonic on $\C_+$, there exists a Herglotz function $f$ such that $\Im f = h$. We denote its Herglotz representation by
\[
f(z) = \alpha z + \beta + \int \left( \frac 1{t-z} - \frac{t}{1+t^2} \right) d\mu(t).
\]
We first compute the constant $\alpha =  \lim_{y\to\infty} \Im f(iy) / y = 1$. Next, by monotone convergence,
\begin{equation}\label{19nov6}
\lim_{y\to \infty} y (\Im f(iy) - y) = \lim_{y\to \infty} \int \frac{y^2}{t^2 + y^2} \,d\mu(t) = \mu(\bbR).
\end{equation}
Thus, \eqref{19nov4} implies $\mu(\bbR) < \infty$, and then \eqref{19nov6} implies \eqref{19nov5} with $a = \mu(\bbR)$.
\end{proof}

\begin{proof}[Proof of Lemma~\ref{lemmamonotonicity}]
The function $M_{\E_1}$ is a positive harmonic function on $\bbC \setminus \E_1$, so it is a positive harmonic function on $\bbC \setminus \E_2$. It has the symmetry $M_{\E_1}(\bar z) = M_{\E_1}(z)$. Thus, applying Lemma~\ref{AL.criterion} with $h= M_{\E_1}$ in the domain $\bbC \setminus \E_2$ shows that $\E_2$ is an Akhiezer--Levin set and $M_{\E_2} \le M_{\E_1}$.

Meanwhile, \eqref{22nov1}, \eqref{eq:posharmGrowth} imply a lower bound  $\lvert \Im z \rvert \le M_{\E_2}(z)$. From the lower and upper bounds and the two-term asymptotics of $M_{\E_1}$ it follows that $M_{\E_2}(iy) = y+ O(1/y)$ as $y\to\infty$, so by Lemma~\ref{expansion.for.linear.h}, $M_{\E_2}$ has the correct two-term asymptotics.

Finally, $h_1 = M_{\E_1} - M_{\E_2}$ is a positive harmonic function on $\bbC\setminus \E_2$ and $h_1(iy) = \frac{b_{\E_1} - b_{\E_2}}y + o(1/y)$, $y \to \infty$. This implies $b_{\E_1} \ge b_{\E_2}$ and, similarly to the proof of Lemma~\ref{expansion.for.linear.h},
\[
h_1(iy) = \int \frac{y}{t^2+y^2} d\mu_1(t), \qquad \mu_1(\bbR) = \lim_{y\to \infty} y h_1(iy) = b_{\E_1} - b_{\E_2}
\]
so $b_{\E_1} = b_{\E_2}$  if and only if $M_{\E_1} = M_{\E_2}$. Theorem~\ref{lem:ConePositive} implies that $M_{\E_1} = M_{\E_2}$ if and only if $\E_2 \setminus \E_1$ is a polar set.
\end{proof}

The set $\bbR \setminus \E_1$ is open, so if it is also polar, it must be empty. Thus for $\E_2 = \bbR$ we obtain the notable special case: if $b_{\E_1} = 0$ then $\E_1 = \bbR$.

\section{Asymptotic behavior of eigensolutions} \label{sectionEigensolutions}

In this section, we study the asymptotic behavior of eigensolutions, beginning with the Weyl solution $\Psi(x, z) = \binom{ \psi_1(x, z) }{ \psi_2(x, z)}$. We denote the $n$-simplex as
\[
\Delta_n(a,b) = \{ t \in \bbR^n \mid a \le t_1 \le t_2 \le \dots \le t_n \le b \}.
\]
Importantly, in the following analysis only integrals over even-dimensional simplices will appear. We therefore denote
\[
I_n(a,b) = \int_{\Delta_{2n}(a,b)} \prod_{k=1}^n \overline{\varphi(t_{2k-1})} e^{2iz(t_{2k} - t_{2k-1})} \varphi(t_{2k}) d^{2n} t.
\]

\begin{lemma}\label{alternate.cpt.support.lemma}
If the operator data $\varphi$ vanishes on $[y_0,\infty)$, then for any $0 \leq a < b \leq y_0 $, a Weyl solution $\Psi$ obeys
\[
\log \frac{\psi_2(a,z)}{\psi_2(b,z)} + i(b-a)z = \log \frac{ 1 + \sum_{n=1}^\infty I_n(a,y_0)}{1 + \sum_{n=1}^\infty I_n(b,y_0)}.
\] 
\end{lemma}
\begin{proof}
For $z \in \C_+$, the function $\binom{ 0 }{ 1 } e^{i z x}$ solves $\Lambda_0 \Psi_0 = z \Psi_0$ and decays exponentially as $x \to +\infty$, so we solve backwards to arrive at a solution to \eqref{eigenequation} on $[0,\infty)$ that decays similarly:
\begin{align*}
	\Psi(x, z) &:= 
	\binom{0}{1}
	e^{i z x} - \int_{x}^{y_0} 
	\begin{pmatrix}
		e^{- i z (x - t_1)} & 0 \\
		0 & e^{i z (x - t_1)}
	\end{pmatrix}
	\begin{pmatrix}
		0 & i \varphi(t_1) \\
		- i \overline{\varphi(t_1)} & 0
	\end{pmatrix}
	\Psi(t_1, z) dt_1.
\end{align*}
Since $\varphi = 0$ on $[y_0, \infty)$, $\Psi$ represents the Weyl solution at $+\infty$, up to normalization. Volterra-type arguments yield a series expansion for $\Psi$, and left multiplication by 
$\begin{pmatrix}
	0 & 1
 \end{pmatrix}$	
gives the following series expansion for $\psi_2(x, z)$:
\begin{align*}
	\psi_2(x, z) &= e^{i z x} \left(1 + \sum_{n = 1}^{\infty} \int_{x}^{y_0} \int_{t_1}^{y_0} \cdots \int_{t_{2 n - 1}}^{y_0} \prod_{k = 1}^{n} \overline{\varphi(t_{2k - 1})} e^{2 i z (t_{2 k} - t_{2 k - 1})} \varphi(t_{2k}) d^{2n} t \right).
\end{align*}
Note that only integrals over even-dimensional simplices arise--this is an artifact of the zero entry in the free Weyl solution $\Psi_0$. We group these integrals in pairs so that later we can take advantage of the fact that $t \in \Delta_{2n}(a,b)$ implies $t_{2k} \geq t_{2k-1}$ for all $k$, which leaves a decaying exponential term in each summand for $z \in \C_+$. Evaluation of $\psi_2$ at $a$ and $b$ followed by taking the logarithm completes the proof. 
\end{proof}

\begin{lemma} \label{shrinking.weyl.disks}
If $U(x,z), V(x,z)$ solve \eqref{eigenequation} for $x \in [a,b+1]$, satisfy \eqref{weyl.disk.condition} at $b + 1$, and ${\Im z > 2\vertiii{\varphi}_1}$, then
\[
\left|\log \frac{u_2(a,z)}{u_2(b,z)} - \log \frac{v_2(a,z)}{v_2(b,z)}\right| \leq 4 \vertiii{\varphi}_1 \exp (-2\Im z + 4\vertiii{\varphi}_1) .
\]
\end{lemma}
\begin{proof}
We introduce the Schur functions for $U$ and $V$ by writing
\begin{equation}
	\log \frac{u_2(a,z)}{u_2(b,z)} - \log \frac{v_2(a,z)}{v_2(b,z)} = -i \int_{a}^{b} \overline{\varphi(t)} (s_U(t,z) - s_V(t,z)) dt. \label{introducing.schur.functions}
\end{equation}
Since for every $t \in [a,b]$, $U$ and $V$ solve \eqref{eigenequation} in $[t, b+1]$ and satisfy \eqref{weyl.disk.condition} at $b + 1$, using the previous lemma and further bounding using $\vertiii{\varphi}_1$, we have
\begin{align*}
	|s_U(t,z) - s_V(t,z)| &\leq 2 \exp(- 2 \Im (z) (b +1 - t) + 2 \lceil b+1-t\rceil \vertiii{\varphi}_1) \\
	&\leq 2 \exp ( (b+1-t)(-2\Im z + 4 \vertiii{\varphi}_1)),
\end{align*}
where in the second bound we use $\frac{\lceil b+1-t \rceil}{b+1-t} \leq 2$ for $t \leq b$. 
Taking the modulus in \eqref{introducing.schur.functions} now gives
\begin{align}
\left\lvert i \int_a^{b} \overline{\varphi(t)} (s_U(t, z) - s_V(t, z)) \, dt \right\rvert &\leq 2 \exp (-2\Im z + 4 \vertiii{\varphi}_1) \int_{a}^{b} |\varphi(t)| e^{(b-t)(-2\Im z + 4\vertiii{\varphi}_1)} dt. \label{using.diff.schur}
\end{align}
Then, splitting the integral $\int_{a}^{b} |\varphi(t)| e^{(b-t)(-2\Im z + 4\vertiii{\varphi}_1)} dt$ into intervals of length one, the condition $\Im z > 2\vertiii{\varphi}_1$ allows us to bound the exponential terms and bound the remaining integrals of $\varphi$ by $\vertiii{\varphi}_1$, which yields
\[
\int_{a}^{b} |\varphi(t)| e^{(b-t)(-2\Im z + 4\vertiii{\varphi}_1)} dt \leq \vertiii{\varphi}_1\Big(\sum_{k=0}^{\lfloor b-a \rfloor - 1} e^{(b-(a+k+1))(-2\Im z + 4\vertiii{\varphi}_1)} + 1\Big). 
\]
Treating the remaining sum as a lower Riemann sum for the function $e^{-(2\Im z - 4\vertiii{\varphi}_1)t}$ yields
\[
\int_{a}^{b} |\varphi(t)| e^{(b-t)(-2\Im z + 4\vertiii{\varphi}_1)} dt \leq  \vertiii{\varphi}_1 \left(\int_{0}^{b-a} e^{-(2\Im z - 4\vertiii{\varphi}_1)t}dt + 1\right) \leq 2 \vertiii{\varphi}_1,
\]
and combining this with \eqref{using.diff.schur} yields the desired bound.
\end{proof}

For Dirac operators on the line, an approximate expansion for the Weyl solution at $+\infty$, $\Psi^+(x,z),$ falls out as a consequence of the previous lemma:

\begin{corollary} \label{epsilon.step.lemma}
If $\varphi \in L^1_{\loc,\unif}(\bbR)$, then for any $a < b$, $\Im z > 2 \vertiii{\varphi}_1$, 
\begin{equation}
\left\lvert \log \frac{\psi_2^+(a,z)}{\psi_2^+(b,z)} + i(b-a)z - \log \frac{ 1 + \sum_{n=1}^\infty I_n(a,b+1)  }{1 + \sum_{n=1}^\infty   I_n(b,b+1)  } \right\rvert \le C
\end{equation}
for some $C$ depending only on $z \in \bbC_+$ and $\vertiii{\varphi}_1$.
\end{corollary}
\begin{proof}
We compare expansions for operator data $\varphi$ and \mbox{$\tilde{\varphi} := \varphi \chi_{[a,b + 1]}$}. Since in the proof of Lemma \ref{alternate.cpt.support.lemma} the point zero had no special significance, the same proof yields, for $a<b$ possibly negative,
\[
	\log \frac{\psi^+_2(a,z, \tilde{\varphi})}{\psi^+_2(b,z, \tilde{\varphi})} + i(b-a)z = \log \frac{ 1 + \sum_{n=1}^\infty I_n(a,b+1)}{1 + \sum_{n=1}^\infty I_n(b,b+1)}.
\]
Since $\Psi^+(x,z, \varphi)$ and $\Psi^+(x,z, \tilde{\varphi})$ each solve \eqref{ZSO} for $x \in [a,b+1]$ and satisfy $\eqref{weyl.disk.condition}$ at $b + 1$, we use the bound from Lemma \ref{shrinking.weyl.disks} and note that this bound depends only on $z \in \C_+$ and $\vertiii{\varphi}_1$.
\end{proof}

Our next goal is to control the series summands $I_n$ towards obtaining a second order approximation of solutions. Note that the control we achieve does not depend on whether we focus on the half-line or full-line case. 

\begin{lemma}\label{bound_on_series_terms}
For any $-\infty<a < b<\infty$,
\[
\lvert I_n(a,b) \rvert \le \frac{ \lceil b-a \rceil^{n} \vertiii{\varphi}_2^{2n} }{(2 \Im z)^n n! } 
 \]
\end{lemma}

\begin{proof}
Define
\[
g(x,y) =  \overline{\varphi(\min(x,y))} e^{-2\Im z \lvert y-x \rvert }  \varphi(\max(x,y))
\]
and define $f: [a,b]^{2n} \to \bbR$ by
\[
f(t) = \prod_{k=1}^n g(t_{2k-1}, t_{2k}).
\]
For any permutation $\pi \in S_n$, define the linear operator $M_\pi$ on $\bbR^{2n}$ by
\[
(M_\pi t)_{2k-1} = t_{2\pi(k) -1}, \quad (M_\pi t)_{2k} = t_{2\pi(k)}.
\]
In words, $M_\pi$ permutes the $n$ pairs $(t_1,t_2),\dots, (t_{2n-1}, t_{2n})$ according to $\pi$. Moreover, for any  sequence of permutations $\epsilon \in S_2^n$, define the linear operator $N_\epsilon$ on $\bbR^{2n}$ by
\[
(N_\epsilon t)_{2k-1} = t_{2k-2+\epsilon(1)}, \quad (N_\epsilon t)_{2k} = t_{2k-2+\epsilon(2)}.
\]
In words, $N_\epsilon$ permutes each pair $(t_{2k-1}, t_{2k})$ according to $\epsilon_k \in S_2$.

For different $\pi, \epsilon$, the products $N_\epsilon M_\pi$ are distinct permutation matrices, so the images of interiors of $\Delta_{2n}(a,b)$ are disjoint subsets of $(a,b)^{2n}$.  Since $g$ is symmetric and $f$ depends in the same way on each pair $(t_{2k-1}, t_{2k})$, for any $\pi, \epsilon$, $f(N_\epsilon M_\pi t) = f(t)$. Thus,
\[
 \int_{\Delta_{2n}(a,b)} \lvert f(t) \rvert d^{2n} t \le \frac 1{2^n n!} \int_{[a,b]^{2n}} \lvert f(t) \rvert d^{2n} t.
\]
Since $f$ is a product of functions which depend only on a single pair $(t_{2k-1}, t_{2k})$, the new integral separates by Tonelli's theorem into a product of $n$ equal integrals, so it remains to prove that
\[
\int_{[a,b]^2}  \lvert \varphi(s) \rvert e^{-2\Im z \lvert t - s \rvert } \lvert \varphi(t) \rvert ds dt \le \frac{\lceil b-a\rceil \vertiii{\varphi}_2^2}{\Im z}.
\]
To prove this, we rewrite the double integral using convolution with the function $w(s) = e^{-2\Im z \lvert s \rvert}$:
\[
\int_{[a,b]^2}  \lvert \varphi(s) \rvert e^{-2\Im z \lvert t - s \rvert } \lvert \varphi(t) \rvert ds dt = \int_{[a,b]} \lvert \varphi(s) \rvert   (w * (\chi_{[a,b]} |\varphi|))(s) ds.
\]
By the Cauchy--Schwarz inequality, this is bounded above by $\|\varphi\|_{L^2(a,b)} \|w * (\chi_{[a,b]} |\varphi|)\|_{L^2(a,b)}$. Young's convolution inequality together with $\|w\|_{L^1(\R)} = \frac{1}{\Im z}$ gives the bound
\[
	\int_{[a,b]^2}  \lvert \varphi(s) \rvert e^{-2\Im z \lvert t - s \rvert } \lvert \varphi(t) \rvert ds dt \le \frac{\lVert \chi_{[a,b]} \varphi \rVert_2^2 }{\Im z} \le \frac{\vertiii{\varphi}^2_2 \lceil b-a \rceil}{\Im z},
\]
which concludes the proof. 
\end{proof}

\begin{lemma}\label{short.hop.lemma}
If $b-a \le 2$ and $\Im z \ge 4 \vertiii{\varphi}_2^2$, then
\[
\left\lvert \log \left( 1 + \sum_{n=1}^\infty I_n(a,b) \right) - I_1(a,b) \right\rvert \le \frac{10 \vertiii{\varphi}_2^4}{(\Im z)^2}.
\]
\end{lemma}

\begin{proof}
Since $ b-a  \le 2$,  Lemma \ref{bound_on_series_terms} implies
\[
\left|\sum_{n=1}^\infty I_n(a,b) \right| \le \sum_{n=1}^\infty \frac{ \vertiii{\varphi}_2^{2n} }{( \Im z)^n n! } = e^{ \vertiii{\varphi}_2^2 / \Im z} - 1.
\]
To bound this further, note that for all $0 < t \le 1/4$ we have $e^t - 1 \le t e^t \le 2t \le 1/2$
(the first inequality follows from convexity of the exponential, the second from $e^{1/4} \le 2$, and the third again from $t \le 1/4$). Now the assumption $\Im z \ge 4 \vertiii{\varphi}_2^2$ implies
\[
\left|\sum_{n=1}^\infty I_n(a,b) \right| \le \frac{2 \vertiii{\varphi}_2^2}{\Im z}.
\]
Use the fact that $\lvert s \rvert \le \frac 12$ implies 
\[
\lvert \log(1+s) - s \rvert \le \sum_{k=2}^\infty \frac {\lvert s\rvert^k}k  \le \frac 12 \sum_{k=2}^\infty \lvert s\rvert^k = \frac{\lvert s \rvert^2}{2( 1-\lvert s \rvert)} \le 2 \lvert s \rvert^2
\]
to bound
\[
\left\lvert \log \left( 1 + \sum_{n=1}^\infty I_n(a,b) \right) - \sum_{n=1}^\infty I_n(a,b) \right\rvert \le \frac{8 \vertiii{\varphi}_2^4}{(\Im z)^2}.
\]
Likewise,
\[
\left\lvert \sum_{n=2}^\infty I_n(a,b) \right\rvert  \le \sum_{n=2}^\infty \frac{ \vertiii{\varphi}_2^{2n} }{( \Im z)^n n! } \le   \frac{\vertiii{\varphi}_2^4}{(\Im z)^2}e^{ \vertiii{\varphi}_2^2 / \Im z} \le \frac{2\vertiii{\varphi}_2^4}{(\Im z)^2}
\]
where the last steps use $e^t - 1 - t \le t^2 e^t$.
\end{proof}

From the previous results, it follows that

\begin{corollary}\label{22aug1_2}
If $\varphi \in L^2_{\loc,\unif}([0,\infty))$, then for any $\delta > 0$, uniformly in $a \ge 0$, 
\begin{equation}
\log \frac{\psi_2(a,z)}{\psi_2(a+1,z)} + iz = I_1(a,a+2) - I_1(a+1,a+2) + O\left( \frac 1{ (\Im z)^2} \right)
\end{equation}
as $z  \to\infty$ with $\delta \le \arg z \le \pi - \delta$.
\end{corollary}

\begin{proof}
Taking $b=a+1$ in Lemma \ref{epsilon.step.lemma} gives the approximation 
\[
	\log \frac{\psi_2(a,z)}{\psi_2(a+1,z)} + iz = \log \frac{ 1 + \sum_{n=1}^\infty I_n(a,a+2)  }{1 + \sum_{n=1}^\infty   I_n(a+1,a+2)  } + O( e^{- 2 \Im z}).
\]
Applying Lemma \ref{short.hop.lemma} gives the claim. 
\end{proof}

We are now ready to give a two term estimate of $\psi_2$ using convolution against the function
\begin{equation} \label{wzdefinition}
w_z(t) = \begin{cases} 2iz e^{-2izt} & t \in [-1,0] \\
0 & \text{else}
\end{cases}
\end{equation}
Note that $\lVert w_z \rVert_{1} \le 1$ for $z = iy$. More generally, $\lVert w_z \rVert_1 \le \lvert z \rvert / \Im z$, so this $L^1$ norm is uniformly bounded in any nontangential cone $\delta \le \arg z \le \pi - \delta$.

\begin{corollary} \label{13jun3}
If $\varphi \in L^2_{\loc,\unif}([0,\infty))$, then for any $\delta > 0$,
\begin{equation} 
\limsup_{x\to \infty} \left\lvert \frac 1x \log \frac{\psi_2(0,z)}{\psi_2(x,z)} + iz  - \frac 1x \frac 1{2iz} \int_0^x \overline{\varphi(t)} (w_z * \varphi)(t) \,dt \right\rvert = O(\lvert z\rvert^{-2}) 
\end{equation}
as $z  \to\infty$ with $\delta \le \arg z \le \pi - \delta$.
\end{corollary}

\begin{proof}
We begin by breaking up the interval $[0,x]$ into intervals of length $1$ so that we can apply Corollary \ref{22aug1_2}. We write,
\[
	\frac 1x \log \frac{\psi_2(0,z)}{\psi_2(x,z)} + i z = \frac 1x \sum_{k=0}^{\lfloor x \rfloor -1} \left(\log \frac{\psi_2(x-k-1, z)}{\psi_2(x-k,z)} + iz\right) + \frac{\{x\}}{x} iz + \frac 1x \log \frac{\psi_2(0,z)}{\psi_2(\{x\}, z)},
\]
where $\lfloor s \rfloor$ denotes the greatest integer less than or equal to $s$, and $\{s\}$ denotes the fractional part of $s$. Because $\psi_2$ is nonzero and continuous, when we take the limit as $x \to \infty$ the latter two terms vanish since $\{x\}<1$. In the sum, using Corollary \ref{22aug1_2} we can rewrite these $\lfloor x \rfloor$ summands as
\begin{align*}
	\log \frac{\psi_2(x-k-1, z)}{\psi_2(x-k,z)} + iz = I_1(x-k-1, x-k+1) - I_1(x-k, x-k+1) + O\left(\frac{1}{(\Im z)^2}\right),
\end{align*}
which, recalling the definition of $I_n(a,b)$, gives the difference of integrals
\[
	\int_{x-k-1}^{x-k+1} \int_{s}^{x-k+1} f(s,t) dtds - \int_{x-k}^{x-k+1} \int_{s}^{x-k+1} f(s,t) dtds = \int_{x-k-1}^{x-k} \int_{s}^{x-k+1} f(s,t) dtds,
\]
where $f(s,t) = \overline{\varphi(s)} e^{2iz(t - s)} \varphi(t)$. The resulting integral is taken over a trapezoidal region such that, as $k$ increases by one, the trapezoid moves down by one and to the left by one in the plane. Taking the sum over $k$ and dividing by $x$ gives 
\[
\frac 1x \sum_{k=0}^{\lfloor x \rfloor -1} \log \frac{\psi_2(x-k-1, z)}{\psi_2(x-k,z)} =\frac 1x \sum_{k=0}^{\lfloor x \rfloor -1} \int_{x-k-1}^{x-k} \int_{s+1}^{x-k+1} f(s,t) dtds + \frac 1x \int_{\{x\}}^{x} \int_{s}^{s+1} f(s,t) dtds + O\left(\frac{1}{(\Im z)^2}\right).
\]
Note that in the remaining sum, the bounds of integration have the variable $t \geq s+1$. In this region $|f(s,t)| \leq e^{-2 \Im z} |\varphi(s)\varphi(t)|$, so that these contributions can be ignored at only exponentially decaying cost, and we write
\begin{align*}
\frac 1x \sum_{k=0}^{\lfloor x \rfloor -1} \log \frac{\psi_2(x-k-1, z)}{\psi_2(x-k,z)} &= \frac 1x  \int_{\{x\}}^{x} \int_{s}^{s+1} f(s,t) dtds + O\left(\frac{1}{(\Im z)^2}\right) \\
	&= \frac 1x \frac{1}{2iz} \int_{\{x\}}^{x} \overline{\varphi(s)} (w_z * \varphi)(s) ds +  O\left(\frac{1}{(\Im z)^2}\right).
\end{align*}
Taking the difference from the statement of the corollary and then the prescribed limit gives the claim.  
\end{proof}

We can now express the growth rate of the Dirichlet solution in the following way:
\begin{theorem}
\label{dirichlet.exp.growth}
If $\varphi \in L^2_{\loc,\unif}([0,\infty))$, then for any $\delta > 0$,
\begin{equation} 
\limsup_{x\to \infty} \left\lvert \frac 1x \log u_1(x,z) + iz  - \frac 1x \frac 1{2iz} \int_0^x \overline{\varphi(t)} (w_z * \varphi)(t) \,dt \right\rvert = O(\lvert z\rvert^{-2}) 
\end{equation}
as $z  \to\infty$ with $\delta \le \arg z \le \pi - \delta$.
\end{theorem}

\begin{proof}
This follows combining Corollary \ref{13jun3} and Lemma \ref{u.times.psi.inequality}.
\end{proof}

This estimate is fundamental to the rest of the paper: in particular the uniformity that follows from the $\limsup$ will allow us to derive two-term expansions of subsequential limits of $\frac 1x \log u_1(x,z)$. We emphasize again that the more complicated second term is needed here: although for any fixed positive $x$,
\[
\frac 1x \log u_1(x,z) + iz  - \frac 1x \frac 1{2iz} \int_0^x \lvert \varphi(t)\rvert^2 \,dt = o(\lvert z \rvert^{-1}),
\]
that statement would not hold in general with a $\limsup_{x\to\infty}$ on the left-hand side.

\section{Subsequential limits} \label{sectionLimits1}

In this section, we begin to investigate the root asymptotics of eigensolutions as $x \to +\infty$. We define
\[
h(x,z) = \frac 1x \log \lvert u_1(x,z) - u_2(x,z) \rvert
\]
and recall that $h$ obeys the symmetry \eqref{hsymmetry}, that is, $h(x,\bar{z}) = h(x,z)$.

\begin{lemma}
\label{pointwise.h.lower.bound}
For any $z\in \bbC_+$,
\[
\liminf_{x\to\infty} h(x,z) \ge 0.
\]
\end{lemma}

\begin{proof}
By Lemma \ref{pass.to.new.h}, it suffices to show that $\liminf_{x\to\infty} \frac{1}{x}\log|u_1(x,z)| \ge 0.$ Since $U$ solves \eqref{eigenequation}, it follows that 
\begin{align*}
\partial_x(U(x,z)^*jU(x,z))=-2 \Im z U(x,z)^*U(x,z).
\end{align*}
Integrating and using $U(x,z)^*jU(x,z)=|u_2(x,z)|^2 - |u_1(x,z)|^2$ and $U(0,z)^*jU(0,z) = 0$ gives
\begin{align*}
|u_1(x,z)|^2=|u_2(x,z)|^2+2\Im z\int_0^xU(s,z)^*U(s,z)ds.
\end{align*}
Note that the integral is increasing in $x$ and positive for any $x_0>0$. Fixing $x_0>0$ gives a constant $C>0$ such that 
\[
	|u_1(x,z)|^2 \geq C\Im z
\]
for all $x > x_0,$ which concludes the proof.
\end{proof}

\begin{lemma}
Let $(a_j, b_j)$ be a gap of $\E=\sigma_{\text{ess}}(\Lambda_\varphi)$, i.e., a connected component of $\bbR \setminus \E$. For any $\epsilon > 0$, the number of zeros of $u_1(x,z)-u_2(x,z)$ in $z \in (a_j + \epsilon, b_j - \epsilon)$ is bounded uniformly in $x \ge 1$.
\end{lemma}

\begin{proof}
The Weyl $M$-matrix centered at $x$, $M(x,z)$, is a matrix-valued Herglotz function that is holomorphic on the complement of the spectrum and has the symmetry $M(x,\bar z)^* = M(x,z)$. From \cite{ClaGes02} we note that it's first diagonal entry $M_{11}$ is given by
\[
	g(x,z) := \frac{-1}{m_-(x,z) + m_+(x,z)},
\]
where $m_+(x,z)$, $m_-(x,z)$ are Weyl functions for restrictions of Dirac operators $L_\varphi$ on $[x,\infty)$ and $[0,x]$, respectively; they are related to Weyl functions for $\Lambda_\varphi$ on the same intervals by the Cayley transform \eqref{eqnCayley}. Thus, $g$ is Herglotz, holomorphic away from the spectrum, and obeys $g(x,\bar{z}) = \overline{g(x,z)}$. For any $\epsilon > 0$, the Dirac operator has a finite number $n_j(\epsilon)$ of eigenvalues on $(a_j+\epsilon,b_j-\epsilon)$. Thus, $M(x,z)$ has $n_j(\epsilon)$ poles in $(a_j + \epsilon, b_j - \epsilon),$ and it follows that $g(x,z)$ has at most $n_j(\epsilon)$ poles in $(a_j + \epsilon, b_j - \epsilon)$. Moreover, every pole of $m_-(x,z)$ is a zero of $g(x,z)$. Since zeros and poles of Herglotz functions strictly interlace, $g$ has at most $n_j(\epsilon)+1$ zeros on $(a_j + \epsilon, b_j - \epsilon),$ so that $m_-(x,z)$ has at most $n_j(\epsilon)+1$ poles on $(a_j + \epsilon, b_j - \epsilon).$ Now, $m_-(x,z)=\infty$ precisely when $s_-(x,z) = 1$, i.e., when $u_1(x,z) - u_2(x,z) = 0$. So $u_1(x,z) - u_2(x,z)$ has at most $n_j(\epsilon)+1$ zeros in $(a_j + \epsilon, b_j - \epsilon)$. 
\end{proof}

\begin{theorem}
\label{precompactness.theorem}
\begin{enumerate}
	\item The family $\mathcal{F}=\{h(x,z)\}_{x \in [1,\infty)}$ is precompact in $\mathcal{D}'(\C)$. 
	\item For any convergent sequence $h(x_j, \cdot)$ in $\mathcal{D}'(\C)$, the limit $h(z) = \lim_{j \to \infty} h(x_j, z)$ is subharmonic on $\C$, harmonic on $\C \setminus \E$, and $h(x_j, \cdot)$ converges to $h(\cdot)$ uniformly on compact subsets of $\C\setminus \E$. 
\end{enumerate}
\end{theorem}
\begin{proof}
Lemma \ref{pointwise.h.lower.bound} gives a lower bound on $h(x,z)$ for each $z \in \C_+$. Using a similar Volterra-type integral expression for $U(x,z)$ as is used for $\Psi(x,z)$ in Lemma \ref{alternate.cpt.support.lemma} and applying Gronwall's inequality gives a uniform upper bound on $h(x,z)$ on compact subsets of $\C$. Lemma~\ref{pointwise.h.lower.bound} gives a pointwise lower bound at an arbitrary point in $\bbC_+$, so by \cite[Theorem 4.1.9]{Hoermander1}, $\mathcal{F}$ is a precompact family in $\mathcal{D}'(\C)$. The rest of the proof follows exactly as in \cite[Theorem 4.3]{EL}, except, of course, that the function $f_j(z)$ should be defined instead as
\[
	f_j(z) = \frac{1}{x_j} \log \left|\frac{u_1(x,z) - u_2(x,z)}{p_j(z)}\right|,
\]
where $p_j$ is the monic polynomial with at most $n_j(\epsilon)$ zeros precisely at the zeros of $u_1 - u_2$.
\end{proof}

We now examine the asymptotics at $\infty$ of subsequential limits of $h(x,z)$.

\begin{theorem} \label{thmharmonicsubseqlimit1}
Let $x_j \to \infty$ be a sequence such that $h_j = h(x_j,\cdot)$ converge in $\cD'(\bbC)$. Then $h= \lim_{j\to\infty} h_j$ defines a positive harmonic function in $\Omega$ and there exists a positive constant $a$ such that $h$ has the asymptotic behavior
\begin{equation}\label{14jun2.2}
h(z) = \Im \Big(z - \frac a{2z}\Big) +o(|z|\inv), 
\end{equation}
as $z \to \infty$, $\arg z \in [\delta, \pi - \delta]$ for any $\delta >0$. Furthermore, the constant $a$ satisfies
\begin{equation}
\label{subsequential.universal.inequality}
a \le \liminf_{j\to\infty} \frac 1{x_j} \int_0^{x_j} \lvert \varphi(t)\rvert^2 \,dt.
\end{equation}
\end{theorem}

\begin{proof}
Harmonicity was proved in Theorem \ref{precompactness.theorem} and positivity in Lemma \ref{pointwise.h.lower.bound}. If we denote
\[
\alpha_j = \frac{1}{x_j} \int_{0}^{x_j} \overline{\varphi(t)} (w_z * \varphi)(t)dt, \qquad a_j =  \frac 1{x_j} \int_0^{x_j} \lvert \varphi(t)\rvert^2 \,dt,  \qquad h_j(z) = h(x_j, z),
\]
then Cauchy-Schwarz together with $\lVert w_z * \varphi \rVert_2 \le \lVert \varphi \rVert_2$ implies that $|\alpha_j| \leq a_j$. With this notation, we can write
\begin{equation}
\label{hj.form}
	h_j(iy) - y = -\frac{\alpha_j}{2y} + \frac{1}{x_j} \log|u_1(x_j,iy)| - y + \frac{\alpha_j}{2y} + \frac{1}{x_j} \log\frac{|u_1(x_j, iy) - u_2(x_j,iy)|}{|u_1(x_j,iy)|}.
\end{equation}
The triangle inequality then gives
\[
	|h_j(iy) -y| \leq \frac{a_j}{2y} + \Big|\frac{1}{x_j} \log|u_1(x_j)| - y + \frac{\frac{1}{x_j} \int_{0}^{x_j} \overline{\varphi} (w_{iy} * \varphi)}{2y}\Big| + \Big|\frac{1}{x_j} \log\frac{|u_1(x_j) - u_2(x_j)|}{|u_1(x_j)|}\Big|
\]
Taking the $\liminf_{j \to \infty}$, on the right hand side we have by Theorem \ref{dirichlet.exp.growth} that the second term is $O(y^{-2})$ and by Lemma \ref{pass.to.new.h} the last term is $o(1)$. Thus we obtain
\[
y - \frac{d}{2y} - \frac C{y^2} \le h(iy) \le y + \frac{d}{2y} +  \frac C{y^2},
\]
as $y \to \infty$, with $d = \liminf_{j\to\infty} a_j$. By Lemma \ref{expansion.for.linear.h}, this implies that $h$ has an expansion \eqref{14jun2.2} and $a \le d$.
\end{proof}

\begin{proof}[Proof of Theorems~\ref{thm11}, \ref{thm12}, and \ref{thm13}]
By the precompactness of the family $\{h(x,z)\}_{x \in [1,\infty)}$ shown in Theorem \ref{precompactness.theorem}, there exists a sequence $x_j \to \infty$ such that $h(x_j,z)=\frac{1}{x_j}\log |u_1(x_j,z) - u_2(x_j,z)|$ converges in $\mathcal{D}'(\C)$ to a function $h(z)$. By Theorem \ref{thmharmonicsubseqlimit1}, $h$ is positive and harmonic in $\Omega$ and $h(iy)/y \to 1$ as $y \to \infty$. This asymptotic behavior implies by Lemma \ref{AL.criterion} that $\Omega$ is Greenian, $\E$ is an Akhiezer-Levin set, and $h \geq M_\E$ in $\Omega$. As a positive harmonic function on $\C_+$, $M_\E(z)$ can be represented as
\[
M_\E(x + iy) = \alpha y + \int \frac{y}{(x-t)^2 + y^2} d\nu(t), \qquad \int \frac{d\nu(t)}{1+t^2} < \infty.
\]
Further, the normalization \eqref{MartinFuncNormalization} implies that $\alpha = 1$, and we see that $M_\E(z) - \Im z$ also defines a positive harmonic function on $\C_+$. Thus $h \geq M_\E$ implies 
\[
M_\E(iy) - y = O(1/y), \qquad y\to \infty.
\]
Lemma \ref{expansion.for.linear.h} completes the proof of Theorem \ref{thm11}. Then, using the expansion for $M_\E(z)$ given by Theorem \ref{thm11} and the fact from Lemma \ref{AL.criterion} that $M_\E(z) \leq h(z)$, the bound \eqref{subsequential.universal.inequality} proves Theorem \ref{thm12}. To prove Theorem \ref{thm13}, we first note that by Lemma \ref{pass.to.new.h} and the fact that 
\[
 |u_1(x,z)| \leq \|U(x,z)\| \leq \sqrt{2}|u_1(x,z)|,
\]
it is enough to show that $M_\E(z) \leq \liminf_{x\to \infty} \frac{1}{x}\log|u_1(x,z)- u_2(x,z)|$ for all $z \in \Omega$. Fix $z_0 \in \C\setminus\R$ and consider a sequence $x_n \to \infty$ such that 
\[
	\lim_{n \to \infty} \frac{1}{x_n} \log |u_1(x_n,z_0) - u_2(x_n,z_0)| = \liminf_{x\to \infty} \frac{1}{x} \log|u_1(x,z_0) - u_2(x,z_0)|.
\]
Passing to a subsequence, we may write $h = \lim_{j \to \infty} \frac{1}{x_{n_j}} \log|u_1(x_{n_j}, z_0) - u_2(x_{n_j}, z_0)|$ and $h \geq M_\E$ in $\Omega$. Thus
\[
	M_\E(z_0) \leq h(z_0) = \liminf_{x\to \infty} \frac{1}{x} \log|u_1(x,z_0) - u_2(x,z_0)|,
\]
which concludes the proof.
\end{proof}

Although Theorem~\ref{thmharmonicsubseqlimit1} allows us to prove Theorems~\ref{thm11}, \ref{thm12}, and \ref{thm13}, it does not give us optimal control of the constant $a$ in the second term. So far, we have used precompactness of $h(x,\cdot)$ in $\cD'(\bbC)$. To refine these results, we will also need to use precompactness of $\sigma_x$ in the vague topology. This will be the subject of the following sections.

\section{Vague convergence} \label{sectionvague}

It is easier to work with weak-$*$ convergence on a compact space, so we begin by reinterpreting some notions from the introduction on the extended real line $\overline{\bbR} = \bbR \cup \{\infty\}$. We  denote by $\cM(\overline{\bbR} )$ the set of positive finite measures on $\overline{\bbR}$, equipped with the weak-$*$ topology dual to $C(\overline{\bbR})$. We will repeatedly use the fact that for any $r > 0$,  the weak-$*$ topology is metrizable on the closed ball $\{ \sigma \in \cM(\overline{\bbR}) \mid \sigma(\overline{\bbR}) \le r \}$ and that this ball is weak-$*$ compact by the Banach--Alaoglu theorem. 

Let us extend the measures $\sigma_x$ by $\sigma_x(\{\infty\}) = 0$ to elements of $\cM(\overline{\bbR} )$, and denote by $\cS'(\varphi)$ the set of sequential weak-$*$ limits of $\sigma_x$ in $\cM(\overline{\bbR})$ as $x \to \infty$.

\begin{lemma}
Let $\vertiii{\varphi}_2 < \infty$. The set $\cS'(\varphi) \subset \cM(\overline{\bbR} )$ is a nonempty compact set.
\end{lemma}

\begin{proof}
For $x \ge 1$,
\[
\sigma_x(\overline{\bbR} ) = \sigma_x(\bbR) \le \frac{\lceil x \rceil}x \vertiii{\varphi}_2^2 \le 2 \vertiii{\varphi}_2^2
\]
so the set of measures $\{ \sigma_x \mid x \ge 1\}$ is a bounded subset of $\cM(\overline{\bbR})$. The rest follows from metrizability and precompactness of bounded sets in $\cM(\overline{\bbR})$.
\end{proof}

\begin{lemma} \label{lemmameasurerestriction}
The sets $\cS(\varphi)$, $\cS'(\varphi)$ are related by
\[
\cS(\varphi) = \{ \sigma\vert_{\bbR} \mid \sigma \in \cS'(\varphi) \},
\]
where $\sigma\vert_{\bbR}$ denotes the restriction of $\sigma$ to the Borel $\sigma$-algebra on $\bbR$.
\end{lemma}

\begin{proof}
If $\sigma_{x_j} \to \sigma \in \cM(\overline{\bbR})$, then since $C_0(\bbR) \subset C(\overline{\bbR})$, $\sigma_{x_j} \to \sigma\vert_{\bbR}$ vaguely. Conversely, assume that $\sigma_{x_j} \to \mu$ vaguely. By precompactness, the sequence $\sigma_{x_j}$ has a subsequence which converges weakly to some $\sigma \in \cS'(\varphi)$, so by the argument from above, it also converges vaguely to $\sigma\vert_\bbR$. It follows that $\mu = \sigma\vert_{\bbR}$.
\end{proof}

Note that we do not claim a bijection between $\cS(\varphi)$ and $\cS'(\varphi)$; the measures in $\cS'(\varphi)$ can contain an additional point mass at $\infty$. In fact, in Example~\ref{xmpl12}, $\cS(\varphi)$ contains only the zero measure on $\bbR$ but $\cS'(\varphi)$ contains different multiples of the Dirac measure at $\infty$.

However, since we are always only interested in the restrictions of those measures on $\bbR$, in all the arguments below, Lemma~\ref{lemmameasurerestriction} allows us to essentially use $\cS(\varphi)$ and $\cS'(\varphi)$ interchangeably. 

\begin{lemma}
The infimum $\inf_{\sigma \in \cS(\varphi)} \sigma(\bbR)$ is a minimum, i.e. there exists $\tilde \sigma \in \cS(\varphi)$ such that
\begin{equation}\label{29aug6}
\inf_{\sigma \in \cS(\varphi)} \sigma(\bbR) = \tilde\sigma(\bbR).
\end{equation}
\end{lemma}

\begin{proof}
Denote the infimum by $C$. There exist $\sigma_n \in \cS(\varphi)$ such that $\sigma_n(\bbR) \le C+ 1/n$. By precompactness, the sequence $\sigma_n$ has a vaguely convergent subsequence with vague limit $\tilde \sigma  \in \cS(\varphi)$.  Since $\bbR$ is open in $\overline{\bbR}$, by the Portmanteau theorem $\tilde\sigma(\bbR) \le \liminf_{n\to\infty} \sigma_n(\bbR) \le C$.
\end{proof}

\begin{lemma} \label{lemma4sep}
$\inf_{\sigma \in \cS(\varphi)} \sigma(\bbR) = \sup_{L > 0} \inf_{\sigma \in \cS(\varphi)} \sigma((-L,L))$.
\end{lemma}

\begin{proof}
One inequality follows immediately from $\sigma(\bbR) \ge \sigma((-L,L))$. To prove the opposite inequality, fix $C > \sup_{L > 0} \inf_{\sigma \in \cS(\varphi)} \sigma((-L,L))$. Then for each $n \in \bbN$ there exists $\sigma_n \in \cS(\varphi)$ such that $\sigma_n((-n,n)) \le C$ for all $n$. There exists a subsequence $\sigma_{n_k}$ which converges vaguely; its limit $\tilde \sigma$ is an element of $\cS(\varphi)$. For any $k \ge j$,
\[
\sigma_{n_k}((-n_j, n_j)) \le  \sigma_{n_k}((-n_k, n_k)) \le C,
\]
so by the Portmanteau theorem, $\tilde \sigma((-n_j, n_j)) \le C$. Since this holds for any $j$, it follows that $\tilde \sigma(\bbR) \le C$, which completes the proof.
\end{proof}

To study the local averages from Theorem~\ref{thmlocalaverages}, we denote
\begin{equation}\label{eqnsinc}
g_\epsilon = \frac{\sqrt{2\pi}}{\epsilon} \chi_{[-\epsilon,0]}, \qquad \widehat g_\epsilon(k) = \frac{e^{i\epsilon k}-1}{i\epsilon k}
\end{equation}
(with $\hat g_\epsilon(0)=1$, of course). To simplify the notation related to convolutions, let us assume that $\varphi$ is also defined for negative $x$ with $\varphi(x) = 0$ for $x < 0$. Then the local averages from Theorem~\ref{thmlocalaverages} can be written in the form
\[
\frac 1x \int_0^x \left\lvert \frac 1{\epsilon} \int_{t}^{t+\epsilon}  \varphi(s)\,ds \right\rvert^2 \,dt =  \frac 1{2\pi x} \int_0^x \left\lvert (g_\epsilon * \varphi)(t)  \right\rvert^2 \,dt.
\]
The difference $\chi_{[0,x]}  (g_\epsilon* \varphi) - g_\epsilon * (\chi_{[0,x]} \varphi)$ is zero outside of $(-\epsilon,\epsilon) \cup  (x-\epsilon,x+\epsilon)$ and bounded by $2\lVert g_\epsilon \rVert_1 \vertiii{\varphi}_1$ there, so we conclude 
\[ 
\frac 1x \int_0^x \left\lvert \frac 1{\epsilon} \int_{t}^{t+\epsilon}  \varphi(s)\,ds \right\rvert^2 \,dt  =   \int \lvert \widehat g_\epsilon \rvert^2\,d\sigma_{x}+ O(x^{-1}), \qquad x \to\infty.
\]
Since $\lVert \widehat g_\epsilon \rVert_\infty =1$ and $\widehat g_\epsilon \to 1$ pointwise  as $\epsilon \downarrow 0$, by  dominated convergence,
\begin{equation}\label{30aug3}
 \sup_{\epsilon > 0}  \int \lvert \widehat g_\epsilon \rvert^2\,d\sigma  = \sigma(\bbR) =  \lim_{\epsilon \downarrow 0}  \int \lvert \widehat g_\epsilon \rvert^2 \,d\sigma.
\end{equation}

The first part of Theorem~\ref{thmlocalaverages} will use the following lemma:

\begin{lemma} \label{lemma181}
\begin{equation}\label{30aug1}
\inf_{\sigma \in \cS(\varphi)} \sigma(\bbR) =  \sup_{\epsilon > 0} \liminf_{x \to\infty}  \int \lvert \widehat g_\epsilon \rvert^2\,d\sigma_{x} = \limsup_{\epsilon \downarrow 0} \liminf_{x \to\infty}  \int \lvert \widehat g_\epsilon \rvert^2\,d\sigma_{x}.
\end{equation}
\end{lemma}

\begin{proof}
Begin by noting that, since $\lvert \widehat g_\epsilon \rvert^2 \in C_0(\bbR)$,
\[
 \liminf_{x \to\infty}  \int \lvert \widehat g_\epsilon \rvert^2\,d\sigma_{x} = \inf_{\sigma \in\cS(\varphi)} \int \lvert \widehat g_\epsilon \rvert^2 \,d\sigma,
 \]
which allows us to rephrase \eqref{30aug1}. The min-max inequality implies that
 \[
   \sup_{\epsilon >  0} \inf_{\sigma \in\cS(\varphi)} \int \lvert \widehat g_\epsilon \rvert^2 \,d\sigma \le   \inf_{\sigma \in\cS(\varphi)}  \sup_{\epsilon > 0}  \int \lvert \widehat g_\epsilon \rvert^2 \,d\sigma = \inf_{\sigma \in \cS(\varphi)} \sigma(\bbR),
\]
and it is obvious that
\[
 \limsup_{\epsilon \downarrow 0} \inf_{\sigma \in\cS(\varphi)} \int \lvert \widehat g_\epsilon \rvert^2 \,d\sigma \le  \sup_{\epsilon > 0}  \inf_{\sigma \in\cS(\varphi)} \int \lvert \widehat g_\epsilon \rvert^2 \,d\sigma.
 \]
It remains to prove
\begin{equation}\label{30aug7}
\inf_{\sigma \in \cS(\varphi)} \sigma(\bbR) \le \limsup_{\epsilon \downarrow 0} \inf_{\sigma \in\cS(\varphi)} \int \lvert \widehat g_\epsilon \rvert^2 \,d\sigma.
\end{equation}
To prove this, note that $\lvert \widehat g_\epsilon \rvert^2$ converges to $1$ uniformly on compacts as $\epsilon \downarrow 0$. Thus, for any $L > 1$, there exists $\delta > 0$ such that for all $\epsilon \in (0,\delta)$, $\lvert \widehat g_\epsilon \rvert^2 \ge (1-L^{-1}) \chi_{(-L,L)}$. Thus,
\[
\limsup_{\epsilon \downarrow 0}  \inf_{\sigma \in\cS(\varphi)} \int \lvert \widehat g_\epsilon \rvert^2 \,d\sigma \ge (1-L^{-1}) \inf_{\sigma \in \cS(\varphi)} \sigma((-L,L)).
\]
Since this holds for arbitrary $L$, Lemma~\ref{lemma4sep} implies \eqref{30aug7}.
\end{proof}

The part of Theorem~\ref{thmlocalaverages} which characterizes regularity will use the following lemma:

\begin{lemma} \label{lemmasupsigma}
\[
\sup_{\sigma \in \cS(\varphi)} \sigma(\bbR)  = \sup_{\epsilon > 0} \limsup_{x \to \infty}  \int \lvert \widehat g_\epsilon \rvert^2\,d\sigma_{x}.
\]
\end{lemma}

\begin{proof}
By definition of $\sigma_x$ and because $\lvert \widehat g_\epsilon \rvert^2 \in C_0(\bbR)$, for every $\epsilon > 0$,
\[
 \limsup_{x\to\infty} \int \lvert \widehat g_\epsilon \rvert^2\,d\sigma_{x} = \sup_{\sigma\in \cS(\varphi)}  \int \lvert \widehat g_\epsilon \rvert^2\,d\sigma.
\]
Thus,
\[
\sup_{\epsilon > 0} \limsup_{x \to \infty} \int \lvert \widehat g_\epsilon \rvert^2\,d\sigma_{x}  = \sup_{\epsilon > 0}  \sup_{\sigma\in \cS(\varphi)}  \int \lvert \widehat g_\epsilon \rvert^2\,d\sigma.
\]
Exchanging the supremums and using \eqref{30aug3} completes the proof.
\end{proof}

\section{Subsequential limits, revisited} \label{sectionLimits2}

We can now replace Theorem~\ref{thmharmonicsubseqlimit1} with a more precise statement. The new statement requires an additional convergence assumption on the subsequence, but it finds the precise two-term asymptotics for the limit. 

\begin{theorem} \label{thmharmonicsubseqlimit2}
Let $x_k \to \infty$ be a sequence such that $h_k = h(x_k, \cdot)$ converge in  $\cD'(\bbC)$ and $\sigma_k = \sigma_{x_k}$ converge vaguely. Then $h= \lim_{k\to\infty} h_k$ defines a positive harmonic function in $\Omega$ and $h$ has the normal asymptotic behavior
\begin{equation}\label{14jun2}
h(iy) = y + \frac {\sigma(\bbR)}{2y} +o(1/y), \quad y\to\infty
\end{equation}
where $\sigma$ is the vague limit of $\sigma_{x_k}$.
\end{theorem}

\begin{proof}
The proof is similar to that of Theorem~\ref{thmharmonicsubseqlimit1}, where we only want to improve the estimate on the second term in the form for $h_k(iy)$. Beginning again with \eqref{hj.form}, instead of using the bound $\|w_{iy}*\varphi\|_2\leq \|\varphi\|_2$, we use $d\sigma_{x_k}$ to rewrite
\begin{align*}
\frac{1}{x_k} \int_0^{x_k} \bar{\varphi} (w	_z * ( \varphi \chi_{[0,x_k]})) & = \frac{1}{x_k}\langle \varphi\chi_{[0,x_k]}, w_z * (\varphi\chi_{[0,x_k]}) \rangle \\
& = \frac{{\sqrt{2\pi}}}{x_k} \langle \widehat{(\varphi\chi_{[0,x_k]})} , \widehat{w_z} \widehat{(\varphi\chi_{[0,x_k]})} \rangle = \sqrt{2\pi} \int \widehat{w_z} d\sigma_{x_k},
\end{align*}
and notice that for each $z \in \C_+$ the difference 
\[
	\frac 1{x_k} \int_0^{x_k} \overline{\varphi(t)} ((w_z * \varphi)(t) - (w_z * \varphi \chi_{[0,x_k]}))dt = \frac{2iz}{x_k} \int_{x_k-1}^{x_k} \overline{\varphi(t)} \int_{x_k-1}^{t} e^{-2iz(t-s)} \varphi(s) dsdt 
\]
goes to zero as $k \to \infty$. Consequently, we write $\frac{1}{x_k} \int_0^{x_k} \bar{\varphi} (w_z * \varphi) = \sqrt{2\pi}  \int \widehat w_z \,d\sigma_{x_k} + O(x_k\inv)$
as $k \to \infty$. Then, we note that $\widehat{w}_z$ has the form
\[
	\widehat w_z (s) = \frac {2iz}{\sqrt{2\pi}} \int_{-1}^0 e^{-2izt} e^{-its} dt =  \frac {2iz}{\sqrt{2\pi}}  \frac{ 1 - e^{i(2z+s)}}{-i(2z+s)} = - \frac 1{\sqrt{2\pi}}  \frac{ 1 - e^{i(2z+s)}}{1 + s/(2z)},
\]
so $\widehat w_z \in C_0(\bbR)$. By vague convergence of $d\sigma_{x_k}$, for any $z\in \bbC_+$,
\[
\frac 1{x_k} \int_0^{x_k} \overline{\varphi(t)} (w_z * \varphi)(t) dt = \sqrt{2\pi}  \int \widehat w_z \,d\sigma_{x_k} + O(x_k\inv) \to  \sqrt{2\pi}  \int \widehat w_z \,d\sigma
\]
as $k\to\infty$. Then, by dominated convergence we have
\[
\lim_{y\to\infty} 2y  \frac{\sqrt{2\pi} \int \widehat w_{iy} \,d\sigma}{2y} = \int 1 d\sigma = \sigma(\bbR),
\]
which concludes the proof.
\end{proof}

\begin{proof}[Proof of Theorem~\ref{thmunivineq2}]
This now follows from the nontangential behavior of $M_\E$, the fact that $h \geq M_\E$ on $\Omega$, and Theorem \ref{thmharmonicsubseqlimit2}.
\end{proof}

Now we are able to prove Theorem~\ref{RegularityTFAE} similarly to its Schr\"odinger analog \cite[Theorem 1.5]{EL}; one important difference is the appearance of the measures $\sigma_x$ and their vague limits, which enter through the expansion \eqref{14jun2}.

\begin{proof}[Proof of Theorem~\ref{RegularityTFAE}]
By inclusions, we have $(v) \implies (iv)$ and $(vi) \implies (iv)$.

$(iv) \implies (vi)$: Consider any sequence $x_k \to \infty$ such that $h=\lim_{k\to\infty} h(x_k,\cdot)$ converges. By Theorem~\ref{thmharmonicsubseqlimit1} and Lemma~\ref{AL.criterion}, the limit obeys $h \ge M_\E$ on $\Omega$. By assumption the opposite inequality holds at some point in $\bbC_+$. By the maximum principle, $h = M_\E$. Thus, $M_\E$ is the only possible subsequential limit. By precompactness with respect to the topology of uniform convergence on compact subsets of $\bbC_\pm$, $h(x,\cdot)$ converges to $M_\E$.

$(vi) \implies (v)$: Denote by $f^\vee$ the upper semicontinuous regularization of $f$. For any $z \in\bbR$ and any sequence $x_n \to \infty$, using $f \le f^\vee$ and the upper envelope theorem \cite[Theorem 2.7.4.1]{Azarin09},
\[
\limsup_{n \to \infty}h(x_n,z) \leq (\limsup_{n \to \infty}h(x_n,z))^\vee = M_\E(z).
\]
 
 $(v) \implies (ii)$: This follows from Theorem~\ref{lem:ConePositive}.

$(ii) \implies (iii)$: The set of Dirichlet irregular points is polar, so is of harmonic measure zero \cite{GarHarmonicMeasure}.

$(iii) \implies (vi)$: Consider a subsequential limit $h(x_n,\cdot) \to h$ for some $x_n \to \infty$. By the upper envelope theorem \cite[Theorem 2.7.4.1]{Azarin09}, there is a polar set $X_1$ such that $\limsup_{n\to\infty} h(x_n,z) = h(z)$ for all $z \in \bbC \setminus X_1$. Meanwhile, by (iii), there is a set $X_2$ of harmonic measure zero such that for all $t \in \E \setminus (X_1 \cup X_2)$ by upper semicontinuity
\[
0 \le \liminf_{\substack{z \to t \\ z \in \Omega}} h(z) \le  \limsup_{\substack{z \to t \\ z \in \Omega}} h(z) \le h(t) \le 0.
\]
Since $X_1 \cup X_2$ has zero harmonic measure, Theorem~\ref{lem:ConePositive} and the leading order asymptotics implies $h = M_\E$. Thus, $M_\E$ is the only subsequential limit and the claim follows by precompactness.

$(vi) \implies (i)$: For any $\sigma \in \cS(\varphi)$, take $x_k \to\infty$ such that $\sigma_{x_k} \to \sigma$. Since $h(x_k, \cdot) \to M_\E$, Theorem~\ref{thmharmonicsubseqlimit2} gives a two-term expansion for $M_\E$; comparing with \eqref{MartinFunc2termexp} implies $\sigma(\bbR) = b_\E$.
 
$(i) \implies (vi)$: Consider a subsequential limit $h(x_k, \cdot) \to h$.  By Theorem \ref{thm11}, the difference $v(z) = h(z) - M_\E(z)$ is  a positive harmonic function on $\Omega$.  By precompactness we can assume $\sigma_{x_k} \to \sigma$ for some $\sigma \in \cS(\varphi)$. Then Theorem \ref{thmharmonicsubseqlimit2} and $\sigma(\bbR) = b_\E$ imply that $v(iy) = o(y\inv)$ as $y \to \infty$. By the proof of Lemma \ref{expansion.for.linear.h}, this implies that $v\equiv 0$, so that $M_\E$ is the only subsequential limit of $h(x,\cdot)$ as $x\to \infty$. By precompactness of the family $\{h(x,\cdot)\}_{x \in [1,\infty)}$, $(vi)$ follows. 
\end{proof}

\begin{proof}[Proof of Theorem~\ref{thmlocalaverages}]
The first claim is a restatement of Theorem~\ref{thmunivineq2}, by Lemma~\ref{lemma181}. The second claim follows from Lemma~\ref{lemmasupsigma}, since regularity is equivalent to $b_\E \le \sup_{\sigma \in \cS(\varphi)} \sigma(\bbR)$. 
\end{proof}

\begin{proof}[Proof of Corollary~\ref{corollarysigmaR}]
By Theorem~\ref{thmlocalaverages}, \eqref{22nov5} implies $b_\E \le 0$, so $b_\E = 0$ and $\E = \bbR$.
\end{proof}

\begin{proof}[Proof of Corollary~\ref{corCesaroNevai}]
This is a special case of Theorem~\ref{thmlocalaverages}, since $b_\E \ge 0$ with equality if and only if $\E = \bbR$.
\end{proof}

\begin{proof}[Proof of Example~\ref{xmpl11}]
It is a calculus fact that the improper integral
\[
\int_0^\infty e^{ix^\alpha} dx = \lim_{L \to \infty} \int_0^L e^{ix^\alpha} dx
\]
is (conditionally) convergent. Subtracting values for $L= t+ \epsilon$ and $L = t$ shows that for any $\epsilon > 0$,
\[
\lim_{t\to\infty} \frac 1{\epsilon} \int_{t}^{t+\epsilon}  \varphi(s)\,ds = 0.
\]
Thus, \eqref{23nov1} holds so by  Corollary~\ref{corCesaroNevai}, $\Lambda_\varphi$ is regular and  $\E = \bbR$; however, $\frac 1x \int_0^x \lvert \varphi(t)\rvert^2 \,dt \to 1$ as $x \to\infty$.
\end{proof}
 
\begin{proof}[Proof of Example~\ref{xmpl12}]
This follows from  Corollary~\ref{corCesaroNevai} similarly to the previous example; however, now $\frac 1x \int_0^x \lvert \varphi(t)\rvert^2 \,dt$ has the set of accumulation points $[\frac 1{q+1}, \frac{q}{q+1}]$.
\end{proof}

\begin{proof}[Proof of Corollary~\ref{corL2precompact}]
The key condition is
\begin{equation}\label{23nov3}
\forall \epsilon > 0 \quad \exists r < \infty \quad \limsup_{x \to \infty} \sigma_x( \bbR \setminus [-r,r]) \le \epsilon.
\end{equation}
Let us first show how the claim follows from this condition. In the conventions of Section~\ref{sectionvague}, it follows that for any subsequential limit $\sigma$, $\sigma(\{\infty\}) \le \epsilon$ for all $\epsilon$, so $\sigma(\{\infty\})=0$. Thus, for any subsequential limit, $\sigma(\bbR) = \lim_{k\to\infty} \sigma_{x_k}(\bbR)$. Thus,
\[
\sup_{\sigma\in \cS(\varphi)} \sigma(\bbR) = \limsup_{x\to\infty} \frac 1x \int_0^x \lvert \varphi(t) \rvert^2 \,dt.
\]
Due to Theorem~\ref{thmunivineq2}, this implies that regularity is equivalent to \eqref{1dec1}.

It remains to prove \eqref{23nov3}. We first note that, if $f_n \to f$ in $L^2((0,1))$ and $t_n \to t$ in $[0,1]$, then $f_n \chi_{(0,t_n)} \to f \chi_{(0,t)}$ in $L^2((0,1))$ because
\[
\lVert f_n \chi_{(0,t_n)} - f \chi_{(0,t)} \rVert \le \lVert (f_n - f) \chi_{(0,t_n)} \rVert + \lVert  ( \chi_{(0,t_n)} - \chi_{(0,t)}) f \rVert.
\]
Thus, precompactness of the family $\{ \varphi(\cdot - x) \vert_{(0,1)} \mid x \ge 0 \}$ implies precompactness of the family
\[
\cP = \{ \varphi(\cdot - x) \chi_{(0,t)} \mid x \ge 0, t \in [0,1] \}
\]
Let us view $\cP$ as a subset of $L^2(\bbR)$ with the standard isometric inclusion of $L^2((0,1))$ into $L^2(\bbR)$. By the Kolmogorov--Riesz--Tamarkin theorem \cite{HHM}, this implies the $L^2$-equicontinuity condition
\[
\lim_{u \to 0} \sup_{x \ge 0} \int  \lvert (\varphi \chi_{(x,x+1)} ) (t) -  (\varphi \chi_{(x,x+1)} ) (t+u) \rvert^2 dt = 0.
\]
By decomposing the interval $(0,x)$ into $\lceil x \rceil$ intervals of length at most $1$, we obtain $\lim_{u \to 0} \omega(u) = 0$ where
\[
\omega(u) = \sup_{x \ge 1} \frac 1x \int   \lvert  (\varphi \chi_{(0,x)} ) (t) -  (\varphi \chi_{(0,x)} ) (t+u) \rvert^2 dt.
\]
By adapting an argument of \cite{Pego} to averages in $x$ instead of a single equicontinuous family, we will obtain a conclusion for Fourier transforms. Denote $\psi(t) = \frac 1{\sqrt{2\pi}} 
e^{-t^2/2}$ and $\psi_R(t) = R \psi(Rt)$. For $\lvert k \rvert \ge 2R$, $1 - \sqrt{2\pi} \hat \psi_R(k) \ge 1/2$ so
\[
\int_{\bbR \setminus [-2R,2R]}  \lvert \widehat{ (\varphi \chi_{(0,x)} ) }(k) \rvert^2 dk \le 4 \lVert (1 - \sqrt{2\pi} \hat \psi_R) \widehat{ (\varphi \chi_{(0,x)} ) } \rVert_2^2 = 4 \lVert \varphi \chi_{(0,x)} - \psi_R * (\varphi \chi_{(0,x)}) \rVert_2^2.
\]
This rewrites as
\[
\int_{\bbR \setminus [-2R,2R]}  \lvert \widehat{ (\varphi \chi_{(0,x)} ) }(k) \rvert^2 dk \le 4 \int \left\lvert \int ((\varphi \chi_{(0,x)})(t) -  (\varphi \chi_{(0,x)})(t+u) ) \psi_R(u) \,du \right\rvert^2 dt.
\]
Since $\psi_R(u) du$ is a probability measure, Jensen's inequality implies
\[
\int_{\bbR \setminus [-2R,2R]}  \lvert \widehat{ (\varphi \chi_{(0,x)} ) }(k) \rvert^2 dk \le 4 \iint \left\lvert (\varphi \chi_{(0,x)})(t) -  (\varphi \chi_{(0,x)})(t+u) \right\rvert^2  \psi_R(u)\, du\,dt.
\]
By Fubini's theorem and the substitution $u=v/R$, this implies
\begin{equation}\label{6dec1}
\frac 1x \int_{\bbR \setminus [-2R,2R]}  \lvert \widehat{ (\varphi \chi_{(0,x)} ) }(k) \rvert^2 dk \le 4 \int \omega(v/R) \psi(v)\, dv.
\end{equation}
Since $\omega(v/R) \to 0$ as $R \to \infty$ and $\lvert \omega(u) \rvert \le 4 \vertiii{\varphi}_2^2$ for all $u$, by dominated convergence, the right-hand side of  \eqref{6dec1} goes to $0$ as $R \to\infty$. Thus, so does the left-hand side, which implies \eqref{23nov3}.
\end{proof}

\begin{proof}[Proof of Theorem~\ref{thm114}]
(a): By Theorem~ \ref{RegularityTFAE} and \ref{RegularityTFAE}, $h(x,z)\to M_\E(z)$ in $\cD'(\bbC)$.
On the other hand, by  definition of $\rho_x$ we have 
\begin{align*}
\frac{1}{2\pi}\Delta h(x,\cdot)=\rho_x.
\end{align*}
Thus, for any $\phi\in C_c^\infty(\bbC)$,
\begin{align*}
\lim\limits_{x\to\infty}\int\phi(t)d\rho_x(t)=\frac{1}{2\pi}\lim\limits_{x\to\infty}\int h(x,t)\Delta \phi(t)d\lambda(t)=\frac{1}{2\pi}\int M_\E(t)\Delta \phi(t)d\lambda(t)=\int \phi(t)d\rho_\E(t).
\end{align*}

(b): Assume that $h(x_n,\cdot)$ do not converge to $M_\E$ and consider a subsequence $x_{n_j}$ such that $h(x_{n_j},\cdot)\to h$ in $\cD'(\C)$ with some limit $h$ not equal to $M_\E$.
By the upper envelope theorem \cite[Theorem 2.7.4.1]{Azarin09} there is a polar set $X_1$ such that for any $z\in \C\setminus X_1$,
\begin{align*}
\limsup\limits_{j\to\infty}h(x_{n_j},z)=h(z).
\end{align*}
The subharmonic function $h$ has some Riesz measure $\rho$ and by the same arguments as in (a), $\rho_{x_{n_j}}$ converges to $\rho$ in the topology dual to $C_c^\infty(\bbR)$.  Hence, by uniqueness of the limits our assumption implies that $\rho=\rho_\E$ and, by Lemma~\ref{lem:positveSymSubharmonic} applied to $h$, there are $d_1,d_2$ so that 
\begin{align*}
h(z)=\int_{-1}^{1}\log|t-x|d\rho_\E(t)+\int_{1\leq|t|}\left(\log\left|1-\frac{z}{t}\right|+\frac{\Re z}{t}\right)d\rho_\E(t)+d_1+d_2\Re z.
\end{align*}
Let $c_1,c_2$ denote the constants appearing the the corresponding representation \eqref{eq:HadamardMartin} for $M_\E$. Since for $\delta<\t<\pi-\delta$
\begin{align*}
0=\lim_{r\to\infty}\frac{M_\E(re^{i\t})-h(re^{i\t})}{r}=\cos(\t)(c_2-d_2),
\end{align*}
we obtain that $d_2=c_2$. Thus, since $M_\E\leq h$ and since we assumed that $h$ is not equal to $M_\E$, we have 
\begin{align*}
0<d=h(z)-M_\E(z).
\end{align*}
Since $M_\E$ has a unique subharmonic extension to $\C$ which vanishes q.e. on $\E$, there is a polar set $X_2$ such that $h(z)=d$ for $z\in\E\setminus X_2$.  In particular,
\begin{align*}
\limsup\limits_{j\to\infty}h(x_{n_j},z)=d>0, \qquad \forall z\in\E\setminus (X_1\cup X_2).
\end{align*} 
However, by Lemma \ref{lemmaSchnol}, for $\mu$-a.e. $z\in \E$, the Dirichlet solution grows at most polynomially and, in particular,
\[
\limsup\limits_{j\to\infty}h(x_{n_j},z)\leq 0.
\]
Thus $\mu(\E\setminus (X_1\cup X_2)) = 0$, which implies the claim with $X = X_1 \cup X_2$.
\end{proof}

With the ingredients developed above, the following proof is analogous to the proof of \cite[Theorem 1.16]{EL}:

\begin{proof}[Proof of Theorem~\ref{IshiiPastur}]
By ergodicity, for a.e. $\eta \in S$, $\lim_{x\to\infty} h(x,z) = \gamma(z)$ pointwise in $\bbC_+ \cup \bbC_-$. Since $\gamma$ is subharmonic in $\bbC$, by the weak identity principle, convergence $h(x,\cdot) \to \gamma$ is also in $\cD'(\bbC)$. By the upper envelope theorem, there is a polar set $X_\eta$ such that for any $z \in \bbC \setminus X_\eta$,
\[
\limsup_{n\to\infty} \frac 1n h(n,z) = \gamma(z).
\]
By Lemma \ref{lemmaSchnol}, $\limsup_{n\to\infty} \frac 1n h(n,z) \le 0$ for $\mu_\eta$-a.e. $z$, si $\mu_\eta(Q \setminus X_\eta) = 0$.
\end{proof}

\begin{proof}[Proof of Theorem~\ref{theoremWidomcriterion}]
By Lemma~\ref{lemmaSchnol}, for $\mu$-a.e.\ $z \in \E$, $\limsup_{x\to \infty} h(x,z) \le 0$. By absolute continuity with respect to $\mu$, this also holds for a.e. $z\in \E$ with respect to harmonic measure. Thus, $\Lambda_\varphi$ is regular by Theorem~\ref{RegularityTFAE}.
\end{proof}

We conclude with an observation: for Schr\"odinger operators, adding a constant to the potential shifts both the average of the potential and the renormalized Robin constant by the same amount. For Dirac operators, there is also a translation invariance of the theory. The substitution $\varphi_1(x) = e^{2icx} \varphi(x)$ shifts the potential by $c$, because $\cU \Lambda_{\varphi} \cU^{-1} = \Lambda_{\varphi_1} + c I$, with the unitary conjugation by
\[
\cU =  \begin{pmatrix} e^{icx} & 0 \\ 0 & e^{-icx}\end{pmatrix}.
\]
This conjugation acts by a translation on the measures $\sigma_x$ and doesn't change any of the averages appearing in the introduction. Likewise, it follows from  \eqref{MartinFunc2termexp} that $b_\E$ is translation-invariant, i.e., $b_{\E + c} = b_\E$ for $c \in \bbR$.

\bibliographystyle{amsplain}

\providecommand{\MR}[1]{}
\providecommand{\bysame}{\leavevmode\hbox to3em{\hrulefill}\thinspace}
\providecommand{\MR}{\relax\ifhmode\unskip\space\fi MR }
\providecommand{\MRhref}[2]{%
	\href{http://www.ams.org/mathscinet-getitem?mr=#1}{#2}
}
\providecommand{\href}[2]{#2}

\end{document}